\documentclass[11pt]{article}

\usepackage{amsmath}
\usepackage{amssymb}
\usepackage{amsfonts}
\usepackage{color}
\usepackage{float}
\usepackage[linktocpage,colorlinks,linkcolor=blue,anchorcolor=blue, citecolor=blue,urlcolor=blue]{hyperref}

\newtheorem{theorem}{Theorem}[section]
\newtheorem{definition}[theorem]{Definition}
\newtheorem{lemma}[theorem]{Lemma}

\newtheorem{corollary}[theorem]{Corollary}
\newtheorem{conjecture}[theorem]{Conjecture}
\newtheorem{proposition}[theorem]{Proposition}

\DeclareMathOperator*{\argmin}{arg\,min}

\newcommand{\BB}{\mathbb{B}}
\newcommand{\C}{\mathbb{C}}
\newcommand{\R}{\mathbb{R}}

\newcommand{\N}{\mathbb{N}}

\newcommand{\BF}{\mathbb{F}\:\!}

\newcommand{\BI}{\mathbb{I}}
\newcommand{\BJ}{\mathbb{J}}

\newcommand{\OO}{\mathcal{O}}

\newcommand{\X}{\mathcal{X}}

\newcommand{\Z}{\mathcal{Z}}

\newcommand{\II}{\mathcal{I}}
\newcommand{\JJ}{\mathcal{J}}
\newcommand{\TT}{\mathcal{T}}

\newcommand{\A}{\mathcal{A}}

\newcommand{\bc}{\boldsymbol{c}}
\newcommand{\be}{\boldsymbol{e}}
\newcommand{\bx}{\boldsymbol{x}}
\newcommand{\by}{\boldsymbol{y}}

\newcommand{\T}{\textnormal{T}}

\newcommand{\tr}{\textnormal{tr}\,}

\newcommand{\rank}{\textnormal{rank}\,}

\newcommand{\sym}{\textnormal{sym}}

\begin{document}

\title{Extreme ratio between spectral and Frobenius norms of nonnegative tensors}

\author{
Shengyu CAO
\thanks{School of Information Management and Engineering, Shanghai University of Finance and Economics, Shanghai 200433, China. Email: shengyu.cao@163.sufe.edu.cn. Research of this author was supported in part by Fundamental Research Funds for Central Universities and National Natural Science Foundation of China under grant 71971132.}
    \and
Simai HE
\thanks{School of Information Management and Engineering, Shanghai University of Finance and Economics, Shanghai 200433, China. Email: simaihe@mail.shufe.edu.cn. Research of this author was supported in part by National Natural Science Foundation of China under grants 71825003 and 72192832.}
    \and
Zhening LI
\thanks{School of Mathematics and Physics, University of Portsmouth, Portsmouth PO1 3HF, United Kingdom. Email: zheningli@gmail.com.}
    \and
Zhen WANG
\thanks{School of Information Management and Engineering, Shanghai University of Finance and Economics, Shanghai 200433, China. Email: zhenwang@163.sufe.edu.cn. Research of this author was supported in part by National Natural Science Foundation of China under grant 72192832.}
}

\date{\today}

\maketitle

\begin{abstract}

One of the fundamental problems in multilinear algebra, the minimum ratio between the spectral and Frobenius norms of tensors, has received considerable attention in recent years. While most values are unknown for real and complex tensors, the asymptotic order of magnitude and tight lower bounds have been established. However, little is known about nonnegative tensors. In this paper, we present an almost complete picture of the ratio for nonnegative tensors. In particular, we provide a tight lower bound that can be achieved by a wide class of nonnegative tensors under a simple necessary and sufficient condition, which helps to characterize the extreme tensors and obtain results such as the asymptotic order of magnitude. We show that the ratio for symmetric tensors is no more than that for general tensors multiplied by a constant depending only on the order of tensors, hence determining the asymptotic order of magnitude for real, complex, and nonnegative symmetric tensors. We also find that the ratio is in general different to the minimum ratio between the Frobenius and nuclear norms for nonnegative tensors, a sharp contrast to the case for real tensors and complex tensors.

\vspace{0.25cm}


\vspace{0.25cm}


%

\end{abstract}

\section{Introduction}

Let $\BF$ be $\C$ (the set of complex numbers), $\R$ (the set of real numbers), $\R_+$ (the set of nonnegative reals), or even a subset of one of these. Given $d$ positive integers $n_1,n_2,\dots,n_d\ge2$, we consider the space $\BF^{n_1\times n_2\times \dots \times n_d}:=\BF^{n_1} \otimes \BF^{n_2} \otimes \dots \otimes\BF^{n_d}$ of tensors of order $d$. One fundamental problem in multilinear algebra is the {\em extreme ratio between the spectral norm and the Frobenius norm} of the space,
\begin{equation}\label{eq:ratio}
\phi(\BF^{n_1\times n_2\times \dots \times n_d}):=\min_{\TT\in\BF^{n_1\times n_2\times \dots \times n_d}\setminus\{\OO\}}\frac{\|\TT\|_\sigma}{\|\TT\|}.
\end{equation}
Here, $\|\TT\|:=\sqrt{\left\langle \TT,\TT\right\rangle}$ denotes the {\em Frobenius norm} (also known as the Hilbert-Schmidt norm), naturally defined by the {\em Frobenius inner product}
$$
\left\langle \TT, \X \right\rangle := \sum_{i_1=1}^{n_1} \sum_{i_2=1}^{n_2} \dots \sum_{i_d=1}^{n_d} \overline{t_{i_1i_2\dots i_d}} x_{i_1i_2\dots i_d} \mbox{ with }\TT=(t_{i_1i_2\dots i_d}), \X=(x_{i_1i_2\dots i_d}),
$$
and $\|\TT\|_\sigma$ denotes the {\em spectral norm}, defined by
\begin{equation}\label{eq:snorm}
\|\TT\|_\sigma :=\max_{\|\bx^k\|=1,\,k=1,2,\dots,d} |\langle \TT, \bx^1\otimes\bx^2\otimes \dots\otimes \bx^d \rangle|,
\end{equation}
where $\bx^k\in\C^{n_k}\mbox{ or }\R^{n_k}$ depending on where $\BF$ resides. The value of $\phi(\BF^{n_1\times n_2\times \dots \times n_d})$ is an attribute of the tensor space and is depended only on the set $\BF$ and the dimensions $n_1,n_2,\dots,n_d$.

Since $\|\TT\|_\sigma\le\|\TT\|$, the maximization counterpart of~\eqref{eq:ratio} is trivially one, obtained by any {\em rank-one tensor} (also called simple tensor), i.e., a tensor that can be written as outer products of vectors such as $\bx^1\otimes\bx^2\otimes \dots\otimes \bx^d$. In this sense, the constant $\phi(\BF^{n_1\times n_2\times \dots \times n_d})$ is the largest coefficient in the norm equivalence inequality, i.e.,
$$
\phi(\BF^{n_1\times n_2\times \dots \times n_d}) \|\TT\| \le \|\TT\|_\sigma \le \|\TT\|.
$$

It is not difficulty to see that $\|\bx^1\otimes\bx^2\otimes \dots\otimes \bx^d\|=1$ if $\|\bx^k\|=1$ for $k=1,2,\dots,d$. By substituting $\bx^1\otimes\bx^2\otimes \dots\otimes \bx^d$ with $\X$ in~\eqref{eq:snorm}, one has $\|\TT\|_\sigma=\max_{\|\X\|=1,\,\rank(\X)=1} |\langle\TT,\X\rangle|$. If we drop the rank-one constraint of the optimization problem, one obtains  $\max_{\|\X\|=1} |\langle\TT,\X\rangle|=\|\TT\|$. Therefore, the constant $\phi(\BF^{n_1\times n_2\times \dots \times n_d})$ measures the gap of this rank-one relaxation from the optimization point of view. Recently, Eisenmann and Uschmajew also considered similar problems for rank-two tensors~\cite{EU21}.

The value of $\phi(\BF^{n_1\times n_2\times \dots \times n_d})$ also originates from an important geometrical fact that the tensor spectral norm measures its approximability by rank-one tensors. To understand this, let $\X(\TT)$ be a best rank-one approximation tensor of $\TT$, i.e., $\X(\TT)$ minimizes $\|\TT-\X\|$ among all rank-one $\X$'s. It is well known (see e.g.,~\cite[Proposition 1.1]{LNSU18}) that $\frac{\X(\TT)}{\|\X(\TT)\|}$ must be an optimal solution to $\max_{\|\X\|=1,\,\rank(\X)=1} |\langle \TT, \X \rangle|=\|\TT\|_\sigma$. Therefore, $$
\phi(\BF^{n_1\times n_2\times \dots \times n_d})=\min_{\TT\in\BF^{n_1\times n_2\times \dots \times n_d}\setminus\{\OO\}}\frac{|\langle \TT, \X(\TT)\rangle|}{\|\TT\|\cdot\|\X(\TT)\|}
$$
and can be seen as the worst-case angle between a tensor and its best rank-one approximation.

The most important notion in quantum mechanics is the quantum entanglement of $d$-partite systems. A $d$-partite state can be represented by a complex tensor $\TT$ of order $d$ with $\|\TT\|=1$. A state $\TT$ is called entangled if it is not a product state (rank-one tensor). One of the quantitative ways to measure the entanglement of a state $\TT$ is the geometric measure of entanglement, given by the distance of $\TT$ to the variety of product states, which is $\sqrt{2(1-\|\TT\|_\sigma)}$. Therefore, the most entangled $d$-partite state is a tensor that achieves the minimum in~\eqref{eq:ratio}, and its geometric measure of entanglement is $\sqrt{2(1-\phi(\C^{n_1\times n_2\times \dots \times n_d}))}$. The readers are referred to~\cite{BFZ22} for the recent development on this topic.

In composition algebras, the value of $\phi(\R^{n_1\times n_2\times \dots \times n_d})$ is directly related to the Hurwitz problem which is to find multiplicative relations between quadratic forms; see~\cite{LNSU18} for details. In algorithm analysis, $\phi(\BF^{n_1\times n_2\times \dots \times n_d})$ governs the convergence rate of truncated steepest descent methods for tensor optimization problems~\cite{U15}. However, in contrast to the various connections and applications mentioned above, this beautiful mathematical problem~\eqref{eq:ratio} has been few studied until the early 2000s by Cobos, K\"uhn, and Peetre~\cite{CKP99,CKP00}. Since Qi~\cite{Q11} formally defined this problem as the best-rank one approximation ratio of a tensor space and proposed several open questions in 2011, there has been a considerate amount of work along this line~\cite{KM15,DFLW17,LNSU18,LZ20,AKU20,EU21,KT22}, especially in the recent a few years.

For $\BF=\C,\R$, or $\R_+$, apart from a trivial case $\phi(\BF^{n_1})=1$ for $d=1$ (vector space) and an easy case $\phi(\BF^{n_1\times n_2})=\frac{1}{\sqrt{\min\{n_1,n_2\}}}$ for $d=2$ (matrix space), the exact values of $\phi(\BF^{n_1\times n_2\times \dots \times n_d})$ are mostly unknown for $d\ge3$. This is mainly due to the NP-hardness to compute the tensor spectral norm~\eqref{eq:snorm} when $d\ge3$~\cite{HLZ10}, let alone the optimization over the spectral norm in~\eqref{eq:ratio}.

For small $n_k$'s, $\phi(\R^{n_1\times n_2\times n_3})$ were determined by K\"uhn and Peetre~\cite{KP06} for all $2\le n_1,n_2,n_3\le 4$ except the case $n_1=n_2=n_3=3$, which was only recently determined by Agrachev, Kozhasov, and Uschmajew~\cite{AKU20}. Many values of $\phi(\R^{n_1\times n_2\times n_3})$ for larger $(n_1,n_2,n_3)$ can be decided by solutions to the Hurwitz problem. These were generalized to $\phi(\BF^{n_1\times n_2\times \dots \times n_d})$ for any order $d$ in the context of orthogonal tensors (for $\BF=\R$) and unitary tensors (for $\BF=\C$)~\cite{LNSU18}. In the complex field, less is understood but the values are usually strict larger than that of the real field from known instances, e.g., $\phi(\C^{2\times 2\times 2})=\frac{2}{3}$~\cite{CKP00} while $\phi(\R^{2\times 2\times 2})=\frac{1}{2}$ and $\phi(\C^{2\times 2\times 2\times  2})=\frac{\sqrt{2}}{3}$~\cite{DFLW17} while $\phi(\R^{2\times 2\times 2\times 2})=\frac{1}{\sqrt{8}}$.

Most efforts in this topic have been put on the lower and upper bounds of $\phi(\BF^{n_1\times n_2\times \dots \times n_d})$ with an aim to establish its asymptotic behaviour when $n_k$'s tend to infinity for fixed $d$. Qi~\cite{Q11} proposed a naive lower bound $\left(\min_{1\le j\le d}\prod_{1\le k\le d,\,k\neq j}n_k\right)^{-\frac{1}{2}}$ of $\phi(\R^{n_1\times n_2\times \dots \times n_d})$, which can be indeed achieved by an interesting class of tensors called orthogonal tensors~\cite{LNSU18}. By applying probabilistic estimates of random tensors in~\cite{TS14}, Li et al.~\cite{LNSU18} showed that
\begin{equation}\label{eq:uppern}
\frac{1}{\sqrt{\min_{1\le j\le d}\prod_{1\le k\le d,\,k\neq j}n_k}} \le \phi(\R^{n_1\times n_2\times \dots \times n_d})
\le \frac{c\sqrt{d\ln d}}{\sqrt{\min_{1\le j\le d}\prod_{1\le k\le d,\,k\neq j}n_k}}
\end{equation}
for some universal constant $c\in\R_+$. A constant $c$ was very recently discovered along with the case of complex field by Kozhasov and Tonelli-Cueto~\cite{KT22}, in which they showed that
$$
\frac{1}{\sqrt{\min_{1\le j\le d}\prod_{1\le k\le d,\,k\neq j}n_k}} \le \phi(\BF^{n_1\times n_2\times \dots \times n_d})
\le \frac{32\sqrt{d\ln d}}{\sqrt{\min_{1\le j\le d}\prod_{1\le k\le d,\,k\neq j}n_k}} \mbox{ if }\BF=\C,\R.
$$
Although nonnegative tensors are more important in real applications, the study of $\phi(\R_+^{n_1\times n_2\times \dots \times n_d})$ remains blank apart from the results implied by $\phi(\R^{n_1\times n_2\times \dots \times n_d})$. In this paper we completely settle its asymptotic behaviour.

The asymptotic behaviour of $\phi(\BF^{n\times n\times \dots \times n})$ was known earlier. By estimating the expectation of the spectral
norm of random tensors, Cobos, K\"uhn, and Peetre~\cite{CKP99} showed that
$$
\frac{1}{n}\le\phi(\R^{n\times n\times n}) \le\frac{3\sqrt{\pi}}{\sqrt{2}n} \mbox{ and }
\frac{1}{n}\le\phi(\C^{n\times n\times n})\le\frac{3\sqrt{\pi}}{n}.
$$
They also remark, but without proof, that
$$
\frac{1}{\sqrt{n^{d-1}}}\le\phi(\R^{n\times n\times\dots \times n})\le\frac{d\sqrt{\pi}}{\sqrt{2n^{d-1}}},
$$
a slightly worse upper bound than that in~\eqref{eq:uppern} by applying $n_k=n$ for $k=1,2,\dots,d$. For nonnegative reals, it was shown recently in~\cite{LZ20} that
$$
\frac{1}{n}\le\phi(\R_+^{n\times n\times n}) \le\frac{1.5}{n^{0.584}},
$$
with the exact order of magnitude remained unclear. However, $\phi(\R_+^{n\times n\times \dots \times n})$ is uncovered with an exact value for even $d$ and an order of magnitude for odd $d$ in this paper.

The extreme ratio for the space of symmetric tensors has attracted particular interest recently~\cite{AKU20,KT22}. A {\em symmetric tensor} is a tensor in $\BF^{n\times n\times\dots \times n}$ and its entries are invariant under permutation of indices. The space of symmetric tensors is denoted by $\BF^{n^d}_\sym$. Since a symmetric tensor in $\BF^{n^d}_\sym$ can be equivalently represented by a homogeneous polynomial function of degree $d$ in $n$ variables, $$\phi(\BF^{n^d}_\sym):=\min_{\TT\in\BF^{n^d}_\sym\setminus\{\OO\}}\frac{\|\TT\|_\sigma}{\|\TT\|}$$
is the same to the minimization of the ratio
between the uniform norm on the unit sphere and the Bombieri norm~\cite{BBEM90} among all homogeneous polynomials of degree $d$ in $n$ variables. Agrachev, Kozhasov, and Uschmajew~\cite{AKU20} showed that the Chebyshev polynomial of degree $d$ is a local minimizer for this optimization problem. However, it is not a global minimizer, disproved by a counterexample in~\cite{LZ20}. Using that example, Li and Zhao~\cite{LZ20} showed that
$$\frac{1}{n}\le\phi(\R^{n^3}_\sym) \le\frac{1.5}{n^{0.584}}.$$
The exact order of magnitude was not clear although we do have $\frac{1}{n} \le \phi(\R^{n\times n\times n})\le \frac{3\sqrt{\pi}}{\sqrt{2}n}$. It is quite obvious that $\phi(\BF^{n\times n\times \dots\times n})\le \phi(\BF^{n^d}_\sym)$. In this paper, by applying a simple idea of homogeneous polynomial mapping, we show that for any $\BF$, $\phi(\BF^{n^d}_\sym)$ is no more than $\phi(\BF^{n\times n\times \dots\times n})$ multiplied by a constant depending only on $d$, nailing down its exact order of magnitude. At the same time, by examining Gaussian tensors, Kozhasov and Tonelli-Cueto~\cite{KT22} in a recent manuscript showed that,
$$
\frac{1}{\sqrt{n^{d-1}}}\le \phi(\BF^{n^d}_\sym) \le\frac{36\sqrt{d!\ln d}}{\sqrt{n^{d-1}}} \mbox{ if }\BF=\C,\R.
$$

The other extreme ratio, dual to $\phi(\BF^{n_1\times n_2\times \dots \times n_d})$, was also studied along with this topic. It is the {\em extreme ratio between the Frobenius norm and the nuclear norm} of a tensor space, i.e.
\begin{equation}\label{eq:ratio2}
\psi(\BF^{n_1\times n_2\times \dots \times n_d}):=\min_{\TT\in\BF^{n_1\times n_2\times \dots \times n_d}\setminus\{\OO\}}\frac{\|\TT\|}{\|\TT\|_*}.
\end{equation}
Here, $\|\TT\|_*$ denotes the nuclear norm, defined by
\begin{equation}\label{eq:nnorm}
\|\TT\|_*:=\min\left\{\sum_{i=1}^r|\lambda_i|: \TT=\sum_{i=1}^r \lambda_i \bx^1_i\otimes\bx^2_i\otimes\dots \otimes\bx^d_i, \|\bx^k_i\|=1\mbox{ for all $k$ and $i$}, r\in\N \right\},
\end{equation}
where $\N$ denotes the set of positive integers. The nuclear norm is the dual norm to the spectral norm and is also NP-hard to compute when $d\ge3$~\cite{FL18}. One obvious fact is $\|\TT\|_\sigma\le\|\TT\|\le\|\TT\|_*$ with equality holds only at rank-one tensors. A perfect result was shown by Derksen et al.~\cite{DFLW17} that
$$
\psi(\BF^{n_1\times n_2\times \dots \times n_d})=\phi(\BF^{n_1\times n_2\times \dots \times n_d}) \mbox{ and } \psi(\BF^{n^d}_\sym)= \phi(\BF^{n^d}_\sym) \mbox{ if }\BF=\C,\R,
$$
and the two extreme ratios can be obtained by the same tensor. This seemingly closed the topic of $\psi$ and left the research to $\phi$. However, for nonnegative reals, $\psi(\R_+^{n_1\times n_2\times \dots \times n_d})$ and $\phi(\R_+^{n_1\times n_2\times \dots \times n_d})$ is in general different, even in different orders or magnitude, to be shown in this paper.

\begin{table}[ht]
    \centering
    \begin{tabular}{|l|l|l|l|l|}
    \hline
    Tensors & $\min_{\TT\neq\OO}\|\TT\|_\sigma/\|\TT\|$ & Reference & $\min_{\TT\neq\OO}\|\TT\|/\|\TT\|_*$ & Reference
\\\hline
    $\C^{n_1\times n_2\times \dots\times n_d}$ & $\max_{j}\prod_{k\neq j}{n_k}^{-\frac{1}{2}}$ & \cite{KT22} & $\max_{j}\prod_{k\neq j}{n_k}^{-\frac{1}{2}}$  & \cite{KT22}+\cite{DFLW17}
\\
    $\R^{n_1\times n_2\times \dots\times n_d}$   & $\max_{j}\prod_{k\neq j}{n_k}^{-\frac{1}{2}}$ & \cite{LNSU18} & $\max_{j}\prod_{k\neq j}{n_k}^{-\frac{1}{2}}$  & \cite{LNSU18}+\cite{DFLW17}
\\
    $\R^{n_1\times n_2\times \dots\times n_d}_+$  & $\prod_{k}{n_k}^{-\frac{1}{4}}$ or $\max_{j}\prod_{k\neq j}{n_k}^{-\frac{1}{2}}$ & Cor.~\ref{thm:upper2}  & $\max_{j}\prod_{k\neq j}{n_k}^{-\frac{1}{2}}$  & Cor.~\ref{thm:nuclear+}
\\\hline
    $\C^{n\times n\times \dots\times n}$ & $n^{-\frac{d-1}{2}}$ & \cite{KT22} & $n^{-\frac{d-1}{2}}$  & \cite{KT22}+\cite{DFLW17}
\\
    $\R^{n\times n\times \dots\times n}$   & $n^{-\frac{d-1}{2}}$ & \cite{CKP99} & $n^{-\frac{d-1}{2}}$ & \cite{CKP99}+\cite{DFLW17}
\\
    $\R^{n\times n\times \dots\times n}_+$   & $n^{-\frac{d}{4}}$ & Cor.~\ref{thm:ordernn} & $n^{-\frac{d-1}{2}}$  & Cor.~\ref{thm:nuclear+}
\\\hline
    $\C^{n^d}_\sym$   & $n^{-\frac{d-1}{2}}$ &  \cite{KT22}, Thm.~\ref{thm:sym} & $n^{-\frac{d-1}{2}}$  & Cor.~\ref{thm:nuclearsym}
\\
    $\R^{n^d}_\sym$ & $n^{-\frac{d-1}{2}}$ & \cite{KT22}, Thm.~\ref{thm:sym} & $n^{-\frac{d-1}{2}}$ & Cor.~\ref{thm:nuclearsym}
\\
    $\R^{n^d}_{+\sym}$ & $n^{-\frac{d}{4}}$ & Cor.~\ref{thm:sym3} & $n^{-\frac{d-1}{2}}$  & Cor.~\ref{thm:nuclear+}
\\\hline
\end{tabular}
\caption{Asymptotic order of magnitude for extreme ratios.} \label{table:list}
\end{table}

We summarize the asymptotic order of magnitude for various cases in the literature together with our own results shown in this paper in Table~\ref{table:list}. Now, let us summarize the main contribution of our work.
\begin{enumerate}
    \item We provide a tight lower bound of $\phi(\R_+{^n_1\times n_2\times\dots\times n_d})$ that can be achieved by a wide class of nonnegative tensors and characterize these extreme tensors.
    \item We provide general lower and upper bounds of $\phi(\R_+^{n_1\times n_2\times\dots\times n_d})$ and $\phi(\R_+^{n\times n\times\dots\times n})$ with either an exact value or an exact order of magnitude.
    \item We show that $\phi(\BF^{n^d}_\sym)$ is no more than $\phi(\BF^{n\times n\times \dots\times n})$ multiplied by a constant depending only on $d$ for any $\BF$ and hence determine the order of magnitude for $\phi(\R^{n^d}_{+\sym})$.
    \item We determine the order of magnitude for $\psi(\R_+^{n_1\times n_2\times \dots \times n_d})$, which is different to that for $\phi(\R_+^{n_1\times n_2\times \dots \times n_d})$, a sharp contrast to the case of $\C$ and $\R$.
    \item We examine $\phi(\R_+^{n_1\times n_2\times n_3})$ for $2\le n_1,n_2,n_3\le 4$ and $\phi(\R^{n^3}_{+\sym})$ for $2\le n\le 4$, providing its exact value or its lower and upper bounds.
\end{enumerate}

The rest of this paper is organized as follows. We first present some uniform notations, tensor operations, and basic properties for tensors and tensor norms in Section~\ref{sec:notation}. We then show the tight lower bound of $\phi(\R_+^{n_1\times n_2\times\dots\times n_d})$ and examine the extreme tensors that achieve the bound in Section~\ref{sec:main}. Finally, we discuss the asymptotic behaviour, symmetric tensors, the extreme ratio between the Frobenius and nuclear norms, as well as low-dimension cases for nonnegative tensors in Section~\ref{sec:other}.

\section{Preparation}\label{sec:notation}

Throughout this paper we uniformly use lowercase letters (e.g., $x$), boldface lowercase letters (e.g., $\bx=\left(x_i\right)$), capital letters (e.g., $X=\left(x_{ij}\right)$), and calligraphic letters (e.g., $\X=\left(x_{i_1i_2\dots i_d}\right)$) to denote scalars, vectors, matrices, and high-order (order $3$ or more) tensors, respectively. We assume that all the dimensions, $n_1,n_2,\dots,n_d$ and $n$, are larger than or equal to two. The convention norm, a norm without a subscript, is the Frobenius norm, which includes the Euclidean norm of vectors as a special case. 

\subsection{Tensor operations}

In order for the tensor operations to be closed in $\BF$, we now only consider $\BF=\C,\R$ or $\R_+$ in this subsection. Nevertheless, these operations can be applied to any $\BF$ in general. A tensor $\TT=(t_{i_1i_2\dots i_d})\in\BF^{n_1\times n_2\times\dots\times n_d}$ has $d$ modes, namely $1,2,\dots,d$. Fixing the mode-$k$ index to $i$ where $1\le i\le n_k$ will result a tensor of order $d-1$ in $\BF^{n_1\times\dots \times n_{k-1}\times n_{k+1}\times\dots\times n_d}$. We call it the $i$th mode-$k$ slice, denoted by $\TT^{(k)}_i$. The {\em mode-$k$ contraction} is obtained by the mode-$k$ product with a vector $\bx=(x_i)\in\BF^{n_k}$, denoted by
$$
\TT\times_k\bx=\sum_{i=1}^{n_k}x_i\TT^{(k)}_i\in \BF^{n_1\times\dots \times n_{k-1}\times n_{k+1}\times\dots\times n_d}.
$$
This is the same mode-$k$ product of a tensor with a matrix widely used in the literature (see e.g.,~\cite{KB09}) by looking at the vector $\bx$ as a $1\times n_k$ matrix. As a consequence, mode contractions by more vectors are obtained by applying mode products one by one, e.g.,
$$
\TT \times_1 \bx \times_2 \by =  (\TT \times_2 \by) \times_1 \bx = (\TT \times_1 \bx) \times_1 \by,
$$
where $\times_1 \by$ in the last equality is used instead of $\times_2 \by$ as mode $2$ of $\TT$ becomes mode $1$ of $\TT \times_1 \bx$. Mode contractions by $d-2$ vectors result a matrix, and with one more contraction result a vector. In particular, one has
\begin{equation} \label{eq:multilinear}
   \TT \times_1 \bx^1 \times_2 \bx^2 \dots \times_d \bx^d
= \langle \TT, \bx^1 \otimes\bx^2\otimes \dots \otimes \bx^d \rangle
=\sum_{i_1=1}^{n_1}\sum_{i_2=1}^{n_2}\dots\sum_{i_d=1}^{n_d}t_{i_1i_2\dots i_d}x^1_{i_1}x^2_{i_2}\dots x^d_{i_d},
\end{equation}
which can be taken as a {\em multilinear form} of $(\bx^1,\bx^2,\dots,\bx^d)$. By multilinearity, it means that it is a linear form of $\bx^j$ by fixing all $\bx^k$'s but $\bx^j$ for every $j=1,2,\dots,d$. Mode contraction by a unit vector will decrease the spectral norm in the weak sense.
\begin{proposition}\label{thm:contraction}
If $\TT\in\BF^{n_1\times n_2\times\dots\times n_d}$ and $\|\bx\|=1$, then $\|\TT\times_k \bx\|_\sigma\le\|\TT\|_\sigma$ for any mode $k$.
\end{proposition}
The proof can be easily obtained from the optimization formulation~\eqref{eq:snorm} since $\langle \TT, \bx^1 \otimes\bx^2\otimes \dots \otimes \bx^d \rangle
= \langle \TT \times_k \bx^k, \bx^2 \otimes \dots\otimes \bx^{k-1}\otimes \bx^{k+1} \otimes\dots\otimes \bx^d \rangle$.

For a fixed mode $k$ and a permutation $\pi=(\pi_1,\pi_2,\dots,\pi_{n_k})$ of $\{1,2,\dots,n_k\}$, a mode-$k$ {\em slice permutation} of $\TT$, is a new tensor in the same size of $\TT$, whose $i$th mode-$k$ slice is $\TT^{(k)}_{\pi_i}$ for every $i$. This is similar to rearranging rows (or columns) of a matrix. For a permutation $\pi=(\pi_1,\pi_2,\dots,\pi_d)$ of $\{1,2,\dots,d\}$, the {\em mode transpose} of $\TT$, denoted by $\TT^\pi\in\BF^{n_{\pi_1}\times n_{\pi_2}\times\dots\times n_{\pi_d}}$, satisfies that
$$
t_{i_1i_2\dots i_d} = (t^\pi)_{i_{\pi_1}i_{\pi_2}\dots i_{\pi_d}} \mbox{ for all }i_1,i_2,\dots,i_d.
$$
In particular, $T^\pi=T^{\T}$ if $T$ is a matrix and $\pi=\{2,1\}$. The following property is obvious.

\begin{proposition}\label{thm:permute}
The spectral, nuclear and Frobenius norms of a tensor are invariant under any slice permutation and mode transpose.
\end{proposition}

Entries of a tensor can be rearranged by combining two modes or splitting a mode. For any two modes of $\TT\in\BF^{n_1\times n_2\times\dots\times n_d}$, say modes $1$ and $2$, a {\em tensor unfolding} of $\TT$ is to combine the two modes into one and result a tensor in $\BF^{n_1n_2\times n_3\times \dots\times n_d}$ of order $d-1$. The reverse operation of tensor unfolding is called {\em tensor folding}. For instance, if $n_1=m_1m_2$ where $m_1,m_2\ge2$ are integers, folding $\TT$ in mode $1$ results a tensor in $\BF^{m_1\times m_2\times n_2\times \dots\times n_d}$ of order $d+1$. Tensor unfoldings can be applied to a tensor repeatedly, so as tensor foldings. In particular, unfolding a tensor $d-2$ times results a matrix, and with one more time results a vector. To the other end, if we let $n_k=\prod_{i=1}^{a_k}p^k_i$ where $2\le p^k_1\le p^k_2\le \dots \le p^k_{a_k}$ are primes for $k=1,2,\dots,d$, the unique tensor of order $\sum_{k=1}^da_k$ with dimension $p^1_1\times p^1_2\times\dots\times p^1_{a_1} \times\dots\times p^d_1\times p^d_2\times\dots\times p^d_{a_d}$ that is folded from $\TT$, is called the {\em maximum folding}.

Given a partition $\{\BI_1,\BI_2,\dots,\BI_s\}$ of the modes $\{1,2,\dots,d\}$, we denote $\TT(\BI_1,\BI_2,\dots,\BI_s)$ to be a tensor of order $s$ with dimension $\prod_{k\in\BI_1}n_k \times \prod_{k\in\BI_2}n_k \times \dots \times \prod_{k\in\BI_s}n_k$, unfolded by combing modes $\BI_k$ of $\TT$ to mode $k$ of $\TT(\BI_1,\BI_2,\dots,\BI_s)$ for $k=1,2,\dots,s$. In particular, if $d$ is even, we call $\TT(\{1,2,\dots,\frac{d}{2}\},\{\frac{d}{2}+1,\frac{d}{2}+2,\dots,d\})$ the {\em standard matricization}. For any mode $1\le k\le d$, we call $\TT(\{k\},\{1,\dots,k-1,k+1,\dots,d\})$ the {\em mode-$k$ matricization}. Also, $\TT(\{1,2,\dots,d\})$ is called the {\em vectorization} of $\TT$, which can be taken as the maximum unfolding. The following monotonicity is quite standard.

\begin{proposition}\label{thm:fold}
If $\TT$ is unfolded to $\X$ (the same to that $\X$ is folded to $\TT$), then
$$
\|\TT\|_\sigma\le\|\X\|_\sigma,
\|\TT\|=\|\X\|,\mbox{ and }
\|\TT\|_*\ge\|\X\|_*.
$$
\end{proposition}
The proof is not difficult by comparing feasibility with optimality from the optimization point of view. We skip it as it needs to introduce many unnecessary notations. One may check~\cite[Proposition 4.1]{WDFS17} for the proof of the spectral norm and apply a similar idea in~\cite[Proposition 4.1]{H15} for the proof of the nuclear norm.

The tensor nuclear norm is the dual norm to the tensor spectral norm described as follows, whose proof can be found in~\cite{LC14}.
\begin{lemma} \label{thm:dual}
Given a tensor $\TT$, one has 
$$
    \|\TT\|_\sigma=\max_{\|\X\|_*\le 1}\langle\TT,\X\rangle \mbox{ and }
    \|\TT\|_*=\max_{\|\X\|_\sigma\le 1}\langle\TT,\X\rangle.
$$
\end{lemma}

\subsection{Symmetric tensor and homogeneous polynomial}\label{sec:sym}

Given a symmetric tensor $\TT\in\BF^{n^d}_\sym$, by substituting $\bx^k=\bx$ for $k=1,2,\dots,d$ in the multilinear form~\eqref{eq:multilinear} one has a homogeneous polynomial function $\langle\TT,\bx\otimes \bx\otimes\dots\otimes \bx\rangle$ of degree $d$ in $n$ variables. A classical result originally due to Banach~\cite{B38} regarding the spectral norm is the following.
\begin{theorem}\label{thm:banach}
If $\TT\in\R^{n^d}_\sym$, then
$$
\|\TT\|_\sigma=\max_{\|\bx^k\|=1,\,k=1,2,\dots,d} |\langle \TT, \bx^1\otimes\bx^2\otimes \dots\otimes \bx^d \rangle|=\max_{\|\bx\|=1} |\langle \TT, \bx\otimes\bx\otimes \dots\otimes \bx \rangle|.
$$
\end{theorem}
In the tensor community, this is known as the best rank-one approximation of a symmetric tensor can be obtained by a symmetric rank-one tensor~\cite{CHLZ12,ZLQ12}.

On the other hand, given any nonzero tensor $\TT\in\BF^{n_1\times n_2\times\dots\times n_d}$, the multilinear form $\langle \TT, \bx^1\otimes\bx^2\otimes \dots\otimes \bx^d \rangle$ itself is a homogeneous polynomial function of degree $d$ in $n=\sum_{k=1}^dn_k$ variables, i.e., $\bx=\left((\bx^1)^{\T},(\bx^2)^{\T},\dots,(\bx^d)^{\T}\right)^{\T}$. Therefore, there is a unique symmetric tensor $\Z\in\BF^{n^d}_\sym$ such that
\begin{equation}\label{eq:homomap}
  \langle\Z, \bx\otimes \bx \otimes \dots \otimes \bx\rangle = \langle\TT, \bx^1\otimes \bx^2 \otimes \dots \otimes \bx^d\rangle.
\end{equation}
From tensor point of view, $\Z$ can be explicitly partitioned into $d^d$ block tensors, which have sizes of $n_{i_1}\times n_{i_2}\times \dots \times n_{i_d}$ where $i_k=1,2,\dots,d$ for $k=1,2,\dots,d$. Among these, there are exactly $d!$ nonzero blocks. Each nonzero block has dimension $n_{\pi_1}\times n_{\pi_2}\times \dots \times n_{\pi_d}$ where $\pi$ is a permutation of $\{1,2,\dots,d\}$ and is equal to $\frac{\TT^\pi}{d!}$ because of~\eqref{eq:homomap}. We remark that this is almost the same idea of symmetric embeddings introduced by Ragnarsson and Van Loan~\cite{RV13} while the connection to homogeneous polynomial is more straightforward. As an example, if $T\in\BF^{n_1\times n_2}$ is a matrix, then $Z=\left(
\begin{array}{cc} O & T/2 \\ T^{\T}/2 & O \end{array}
\right)$ while the symmetric embedding of $T$ makes it $\left(
\begin{array}{cc} O & T \\ T^{\T} & O \end{array}
\right)$. We shall use this idea to study the extreme ratio for symmetric tensors in Section~\ref{sec:symmetric}.

\subsection{Basic properties of extreme ratios}

We provide some properties regarding the extreme ratio between the spectral and Frobenius norms and that between the Frobenius and nuclear norms. The first two results are immediate from Proposition~\ref{thm:permute} and Proposition~\ref{thm:fold}, respectively.

\begin{lemma}\label{thm:permute2}
The ratio between the spectral and Frobenius norms of a nonzero tensor is invariant under slice permutation, mode transpose and multiplication by a nonzero constant. This is the same to the ratio between the Frobenius and nuclear norms.
\end{lemma}

\begin{lemma}\label{thm:fold2}
If $\BF_1$ and $\BF_2$ are two spaces where tensors in $\BF_2$ are obtained by unfolding tensors in $\BF_1$, then $\phi(\BF_1)\le \phi(\BF_2)$ and $\psi(\BF_1)\le \psi(\BF_2)$.
\end{lemma}

Our final property is on the monotonicity of the extreme ratios with respect to the dimensions.
\begin{lemma}\label{thm:mono}
If $n_k\le m_k$ for $k=1,2,\dots,d$, then
\begin{align*} \phi(\BF^{m_1\times m_2\times\dots\times m_d}) &\le \phi(\BF^{n_1\times n_2\times\dots\times n_d}), \\
 \psi(\BF^{m_1\times m_2\times\dots\times m_d}) &\le \psi(\BF^{n_1\times n_2\times\dots\times n_d}).
\end{align*}
For any positive integer $m$ and mode $k$ where $1\le k\le d$, one has
\begin{align*}
    \phi(\BF^{n_1\times n_2\times\dots\times n_d})&\le
{\sqrt{m}}\,\phi(\BF^{n_1\times \dots \times n_{k-1}\times mn_k \times n_{k+1}\times \dots\times n_d}), \\
\psi(\BF^{n_1\times n_2\times\dots\times n_d})& \le
{\sqrt{m}}\,\psi(\BF^{n_1\times \dots \times n_{k-1}\times mn_k \times n_{k+1}\times \dots\times n_d}).
\end{align*}
\end{lemma}
\begin{proof}
The first two bounds are trivial as $\BF^{n_1\times n_2\times\dots\times n_d}$ can be taken as a subset of $\BF^{m_1\times m_2\times\dots\times m_d}$ by enlarging the dimensions with zero entries.

For any $m\in\N$, assume without loss of generality that $k=1$. Let $\TT\in\BF^{mn_1\times n_2\times\dots\times n_d}$, which can be partitioned into $m$ block tensors in $\BF^{n_1\times n_2\times\dots\times n_d}$ via cuts in mode $1$. Denote these block tensors to be $\TT_1,\TT_2,\dots,\TT_m$. Let $\|\TT_i\|=\max_{1\le j\le m}\|\TT_j\|$. According to the bounds for the spectral norm of subtensors~\cite[Theorem 3.1]{L16}, one has $\|\TT_i\|_\sigma\le\|\TT\|$. Therefore,
$$
\frac{{\|\TT\|_\sigma}^2}{\|\TT\|^2}\ge \frac{{\|\TT_i\|_\sigma}^2}{\sum_{j=1}^m\|\TT_j\|^2}\ge \frac{{\|\TT_i\|_\sigma}^2}{m\|\TT_i\|^2}\ge \frac{\phi(\BF^{n_1\times n_2\times\dots\times n_d})^2}{m}.
$$
By the generality of $\TT$, we have $\phi(\BF^{mn_1\times n_2\times\dots\times n_d})^2\ge \frac{\phi(\BF^{n_1\times n_2\times\dots\times n_d})^2}{m}$, which shows the third bound. The last bound can be shown similarly as the third.
\end{proof}

\section{Extreme ratio between spectral and Frobenius norms} \label{sec:main}

In this section, we provide an almost complete picture of the tight lower bound of the extreme ratio between the spectral and Frobenius norms for nonnegative tensors. The lower bound can be obtained by a wide class of $n_k$'s that can tend to infinity. Our main result is as follows.

\begin{theorem}\label{thm:main}
Let $d$ positive integers $n_1,n_2,\dots,n_d\ge2$ for $\R_+^{n_1\times n_2\times\dots\times n_d}$.
\begin{enumerate}
\item The extreme ratio between the spectral and Frobenius norms
  \begin{equation}\label{eq:main}
    \phi(\R_+^{n_1\times n_2\times\dots\times n_d})
    =\min_{\TT\in\R_+^{n_1\times n_2\times\dots\times n_d}\setminus\{\OO\}} \frac{\|\TT\|_\sigma}{\|\TT\|}
    \ge \left(\prod_{k=1}^dn_k\right)^{-\frac{1}{4}}.
  \end{equation}
\item The lower bound is attained if and only if $\sqrt{\prod_{k=1}^dn_k}$ is an integer that can be divided by every $n_k$, i.e.,
\begin{equation}\label{eq:condition}
\frac{\sqrt{\prod_{k=1}^dn_k}}{n_k}\in\N \mbox{ for }k=1,2,\dots,d.
\end{equation}
\item The lower bound is achieved by an unfolded identity tensor, up to slice permutation and multiplication by a positive constant, and is attained if and only if an unfolded identity tensor exists in $\R_+^{n_1\times n_2\times\dots\times n_d}$.
\end{enumerate}
\end{theorem}

We shall prove the theorem in a discussion style, starting with the lower bound in Section~\ref{sec:lower}, from which the condition of equality is derived. We then propose the nonnegative tensors that obtain the lower bound under this condition, i.e., the concept of unfolded identity tensors in Section~\ref{sec:sit}. Finally  we generalize unfolded identity tensors with an aim to fully characterize these extreme tensors under this condition in Section~\ref{sec:picutre}.

\subsection{Lower bound and necessary condition} \label{sec:lower}

\begin{lemma}
$
\phi(\R_+^{n_1\times n_2\times \dots\times n_d})\ge \left(\prod_{k=1}^dn_k\right)^{-\frac{1}{4}}.
$
\end{lemma}
\begin{proof}
First, one has $\phi(\R_+^{n_1\times n_2\times \dots\times n_d})=\min_{\|\TT\|=1}\|\TT\|_\sigma$ and so $\frac{1}{\phi(\R_+^{n_1\times n_2\times \dots\times n_d})}=\max_{\|\TT\|_\sigma=1}\|\TT\|$. Let us take a close look at the optimization problem $\max_{\|\TT\|_\sigma=1}\|\TT\|$. Since $\|\TT\|_\sigma=1$, one obviously has $0\le t_{i_1i_2\dots i_d}\le 1$ for any entry $t_{i_1i_2\dots i_d}$ of $\TT$. Moreover, as $\left\|\frac{\be_k}{\sqrt{n_k}}\right\|=1$ for any $1\le k\le d$ where $\be_k\in\R^{n_k}$ is an all-one vector, by~\eqref{eq:multilinear} one has
\begin{equation} \label{eq:upper}
    \frac{1}{\sqrt{\prod_{k=1}^dn_k}}\sum_{i_1=1}^{n_1}\sum_{i_2=1}^{n_2}\dots\sum_{i_d=1}^{n_d}t_{i_1i_2\dots i_d}
=\left\langle \TT, \frac{\be_1}{\sqrt{n_1}}\otimes \frac{\be_2}{\sqrt{n_2}}\otimes\dots\otimes\frac{\be_d}{\sqrt{n_d}}\right\rangle
\le \|\TT\|_\sigma=1.
\end{equation}
This leads to
$$
\|\TT\|
= \left(\sum_{i_1=1}^{n_1}\sum_{i_2=1}^{n_2}\dots\sum_{i_d=1}^{n_d}{t_{i_1i_2\dots i_d}}^2\right)^{\frac{1}{2}}
\le \left(\sum_{i_1=1}^{n_1}\sum_{i_2=1}^{n_2}\dots\sum_{i_d=1}^{n_d}t_{i_1i_2\dots i_d}\right)^{\frac{1}{2}}
\le \left(\sqrt{\prod_{k=1}^dn_k}\right)^{\frac{1}{2}}
= \left(\prod_{k=1}^dn_k\right)^{\frac{1}{4}},
$$
where the first inequality is due to $0\le t_{i_1i_2\dots i_d}\le 1$ and the second inequality is due to~\eqref{eq:upper}. This shows that
$\max_{\|\TT\|_\sigma=1}\|\TT\|\le \left(\prod_{k=1}^dn_k\right)^{\frac{1}{4}}$. Therefore,
$$
\phi(\R_+^{n_1\times n_2\times \dots\times n_d}) = \min_{\|\TT\|=1}\|\TT\|_\sigma = \frac{1}{\max_{\|\TT\|_\sigma=1}\|\TT\|} \ge \left(\prod_{k=1}^dn_k\right)^{-\frac{1}{4}}.
$$
proving the lower bound.
\end{proof}

From the above proof, if the lower bound $\left(\prod_{k=1}^dn_k\right)^{-\frac{1}{4}}$ is obtained at $\TT$, and further if we only consider $\|\TT\|_\sigma=1$ (if not we can scale it), then $\TT\in\BB^{n_1\times n_2\times \dots\times n_d}$ where $\BB=\{0,1\}$ since ${t_{i_1i_2\dots i_d}}^2=t_{i_1i_2\dots i_d}$ for any entry $t_{i_1i_2\dots i_d}$. Moreover, the number of nonzero entries of $\TT$ must be $\sqrt{\prod_{k=1}^dn_k}$ because~\eqref{eq:upper} must be held as an equality. This obviously implies that $\sqrt{\prod_{k=1}^dn_k}$ is an integer. In fact, these nonzero entries must be evenly distributed among slices.

\begin{proposition}\label{thm:evenly}
Let $\TT\in\R_+^{n_1\times n_2\times \dots\times n_d}$ with $\|\TT\|_\sigma=1$ and $\|\TT\|=\left(\prod_{k=1}^dn_k\right)^{\frac{1}{4}}$. One has
\begin{enumerate}
\item $\TT\in\BB^{n_1\times n_2\times \dots\times n_d}$ with $\sqrt{\prod_{k=1}^dn_k}$ nonzero entries.
\item Any mode-$k$ slice of $\TT$ has $\frac{\sqrt{\prod_{k=1}^dn_k}}{n_k}$ number of nonzero entries for $k=1,2,\dots,d$.
\end{enumerate}
\end{proposition}
\begin{proof}
From the previous discussion, it suffices to show that all mode-$k$ slices must have the same number of nonzero entries since the number of mode-$k$ slices is $n_k$ and the total number of nonzero entries is $\sqrt{\prod_{k=1}^dn_k}$. Without loss of generality, we only show this for mode-$1$ slices.

Let $m_j=\sum_{i_2=1}^{n_2}\sum_{i_3=1}^{n_3}\dots\sum_{i_d=1}^{n_d}t_{ji_2i_3\dots i_d}$, the number of nonzero entries of $\TT^{(1)}_j$ (the $j$th mode-$1$ slice of $\TT$) for $j=1,2,\dots,n_1$. Since $\left\|\frac{\be_k}{\sqrt{n_k}}\right\|=1$ for $k=2,3,\dots,d$, we have
$$
\frac{1}{\sqrt{\prod_{k=2}^dn_k}} \left\|(m_1,m_2,\dots,m_{n_1})\right\|
= \left\|\TT\times_2 \frac{\be_2}{\sqrt{n_2}} \dots \times_d \frac{\be_d}{\sqrt{n_d}}\right\| \le\|\TT\|_\sigma =1,
$$
where the inequality is obtained by applying Proposition~\ref{thm:contraction} $d-1$ times. As a result,
$$
\sqrt{\frac{\sum_{j=1}^{n_1}{m_j}^2}{n_1}} = \frac{\left\|(m_1,m_2,\dots,m_{n_1})\right\|}{\sqrt{n_1}} \le \frac{\sqrt{\prod_{k=2}^dn_k}}{\sqrt{n_1}} = \frac{\sqrt{\prod_{k=1}^dn_k}}{n_1} = \frac{\sum_{j=1}^{n_1}m_j}{n_1},
$$
where the last equality holds because the number of nonzero entries of $\TT$ is $\sqrt{\prod_{k=1}^dn_k}$. According to the generalized mean inequality $\sqrt{\frac{\sum_{j=1}^{n_1}{m_j}^2}{n_1}}\ge\frac{\sum_{j=1}^{n_1}m_j}{n_1}$, the above must hold at the equality with $m_1=m_2=\dots=m_{n_1}$.
\end{proof}

For any $\TT\in\R_+^{n_1\times n_2\times \dots\times n_d}$ that obtains the extreme ratio $\left(\prod_{k=1}^dn_k\right)^{-\frac{1}{4}}$ in~\eqref{eq:main}, one can certainly multiply a positive constant to make $\|\TT\|_\sigma=1$. Thus, Proposition~\ref{thm:evenly} immediately implies the necessity of condition~\eqref{eq:condition} in Theorem~\ref{thm:main} as the number of nonzero entries of any mode-$k$ slice, $\frac{\sqrt{\prod_{k=1}^dn_k}}{n_k}$, must be a positive integer. On the other hand, we are indeed able to construct a zero-one tensor that obtains the extreme ratio $\left(\prod_{k=1}^dn_k\right)^{-\frac{1}{4}}$ under~\eqref{eq:condition}.

\subsection{Unfolded identity tensor}\label{sec:sit}

We study zero-one tensors that achieve the lower bound~\eqref{eq:main} in Theorem~\ref{thm:main}, called unfolded identity tensors. To start with, let us consider the identity matrix $I_n\in\BB^{n\times n}$. Let $n=\prod_{k=1}^{s_n} p_k$ be the prime factorization where $2\le p_1\le p_2\le \dots \le p_{s_n}$. Let $\II_n\in\BB^{p_1\times p_2\times \dots \times p_{s_n}\times p_1 \times p_2\times\dots\times p_{s_n}}$ be the maximum folding of $I_n$, called the {\em $n$th identity tensor}. This is a tensor of order $2s_n$ whose standard matricization is $I_n$, i.e.,
$$
\II_n\left(\{1,2,\dots,s_n\},\{s_n+1,s_n+2,\dots,2s_n\}\right)=I_n.
$$
It is easy to see that
$$
\|\II_n\|=\|I_n\|=\sqrt{n} \mbox{ and } 1\le\|\II_n\|_\sigma\le\|I_n\|_\sigma=1,
$$
since $\II_n$ is folded by $I_n$ by Proposition~\ref{thm:fold}. Obviously $\II_n$ is unique for any given $n$. 

\begin{definition} \label{def:uit}
Given any partition $\{\BI_1, \BI_2,\dots, \BI_d\}$ of modes $\{1,2,\dots,2s_n\}$ that satisfies
\begin{equation}\label{eq:condition2}
    \left|\BI_k\bigcap\{j,s_n+j\}\right|\le1 \mbox{ for } 1\le k\le d \mbox{ and } 1\le j\le s_n,
\end{equation}
$\II_n(\BI_1, \BI_2,\dots, \BI_d)$ is called an unfolded identity tensor (UIT).
\end{definition}
According to the definition, any mode transpose of a UIT is already included in Definition~\ref{def:uit} via permutation of the $\BI_k$'s. Also, $\II_n$ itself, as well as its mode transposes, is a UIT.

A key property of a UIT is that its spectral norm is one, albeit it is unfolded from $\II_n$ whose spectral norm is also one. Before computing the spectral norm of a UIT, we need a technical result that has an independent interest.
\begin{lemma}\label{thm:product2}
Let integer $d\ge2$ and $\X^1,\X^2,\dots,\X^d$ be tensors with appropriate dimensions of order $a_1,a_2,\dots,a_d$, respectively, satisfying that $a_1+a_2+\dots+a_d=2s$ is even. If $\BI_k=\{j^k_1,j^k_2,\dots,j^k_{a_k}\}\subseteq\{1,2,\dots,s\}$ for $k=1,2,\dots,d$ and there are exactly two $\BI_k$'s containing $j$ for every $j=1,2,\dots,s$, then
\begin{equation} \label{eq:product}
\sum_{i_1,i_2,\dots,i_{s}} x^{1}_{i_{j^1_1} i_{j^1_2}\dots i_{j^1_{a_1}}} x^{2}_{i_{j^2_1} i_{j^2_2}\dots i_{j^2_{a_2}}} \dots x^{d}_{i_{j^d_1} i_{j^d_2} \dots i_{j^d_{a_d}}} \le \prod_{k=1}^d\|\X^k\|,
\end{equation}
where the summand for $i_j$ (that appears exactly twice in subscripts of $x$'s) runs from $1$ to an appropriate value under appropriate dimensions of these $\X^k$'s for every $j=1,2,\dots,s$.
\end{lemma}
\begin{proof}
The proof is based on the induction on $d$. For $d=2$, we simply have $\BI_1=\BI_2=\{1,2,\dots,s\}$. The summand in~\eqref{eq:product} makes $\left\langle \X^1,\X^2\right\rangle$ under a possible mode transpose of $\X^2$, which is less than or equal to $\|\X^1\|\cdot\|\X^2\|$ according to Cauchy-Schwarz inequality.

For general $d\ge3$, let $\BI_1\bigcap\BI_2=\{q_1,q_2,\dots,q_r\}$. Without loss of generality, we may denote $\BI_1=\{j^1_1,j^1_2,\dots,j^1_{b_1},q_1,q_2,\dots,q_r\}$ and $\BI_2=\{j^2_1,j^2_2,\dots,j^2_{b_2},q_1,q_2,\dots,q_r\}$ where $b_1=a_1-r$ and $b_2=a_2-r$. Let us consider a new tensor $\Z$ of order $b_1+b_2$, whose $(i_{j^1_1}, i_{j^1_2},\dots,i_{j^1_{b_1}},i_{j^2_1},i_{j^2_2},\dots,i_{j^2_{b_2}})$th entry is defined by
$\sum_{i_{q_1},i_{q_2},\dots,i_{q_r}} x^{1}_{i_{j^1_1} i_{j^1_2} \dots i_{j^1_{b_1}} i_{q_1} i_{q_2} \dots i_{q_r}} x^{2}_{i_{j^2_1} i_{j^2_2} \dots i_{j^2_{b_2}} i_{q_1} i_{q_2} \dots i_{q_r}}$. We have
\begin{align*}
\|\Z\|^2 & = \sum_{i_{j^1_1},i_{j^1_2},\dots,i_{j^1_{b_1}},i_{j^2_1},i_{j^2_2},\dots,i_{j^2_{b_2}}} \left( \sum_{i_{q_1},i_{q_2},\dots,i_{q_r}} x^{1}_{i_{j^1_1} i_{j^1_2}\dots i_{j^1_{b_1}} i_{q_1} i_{q_2} \dots i_{q_r}} x^{2}_{i_{j^2_1} i_{j^2_2} \dots i_{j^2_{b_2}} i_{q_1} i_{q_2} \dots i_{q_r}} \right)^2 \\
&\le \sum_{i_{j^1_1},i_{j^1_2},\dots,i_{j^1_{b_1}},i_{j^2_1},i_{j^2_2},\dots, i_{j^2_{b_2}}}  \sum_{i_{q_1},i_{q_2},\dots,i_{q_r}} \left(x^{1}_{i_{j^1_1} i_{j^1_2}\dots i_{j^1_{b_1}} i_{q_1} i_{q_2} \dots i_{q_r}}\right)^2 \sum_{i_{q_1},i_{q_2},\dots,i_{q_r}} \left(x^{2}_{i_{j^2_1} i_{j^2_2} \dots i_{j^2_{b_2}} i_{q_1} i_{q_2} \dots i_{q_r}} \right)^2  \\
&= \sum_{i_{j^1_1},i_{j^1_2},\dots,i_{j^1_{b_1}}}  \sum_{i_{q_1},i_{q_2},\dots,i_{q_r}} \left(x^{1}_{i_{j^1_1} i_{j^1_2}\dots i_{j^1_{b_1}} i_{q_1} i_{q_2} \dots i_{q_r}}\right)^2 \sum_{i_{j^2_1},i_{j^2_2},\dots, i_{j^2_{b_2}}}  \sum_{i_{q_1},i_{q_2},\dots,i_{q_r}} \left(x^{2}_{i_{j^2_1} i_{j^2_2} \dots i_{j^2_{b_2}} i_{q_1} i_{q_2} \dots i_{q_r}} \right)^2  \\
&=\|\X^1\|^2\|\X^2\|^2,
\end{align*}
where the inequality is due to Cauchy-Schwarz inequality.

Since $q_1,q_2,\dots,q_r$ belong to both $\BI_1$ and $\BI_2$, none of them belongs to any $\BI_k$ for $k\ge3$. Hence the summand for $i_{q_1},i_{q_2},\dots,i_{q_r}$ in~\eqref{eq:product} is irrelevant to $\X^3,\X^4,\dots,\X^d$. Thus, the summand in~\eqref{eq:product} can be rewritten as
$$
\sum_{\{i_1,i_2,\dots,i_{s}\}\setminus\{i_{q_1},i_{q_2},\dots,i_{q_r}\}} x^{3}_{i_{j^3_1} i_{j^3_2}\dots i_{j^3_{a_3}}} \dots x^{d}_{i_{j^d_1} i_{j^d_2} \dots i_{j^d_{a_d}}} \sum_{i_{q_1},i_{q_2},\dots,i_{q_r}} x^{1}_{i_{j^1_1} i_{j^1_2} \dots i_{j^1_{b_1}} i_{q_1} i_{q_2} \dots i_{q_r}} x^{2}_{i_{j^2_1} i_{j^2_2} \dots i_{j^2_{b_2}} i_{q_1} i_{q_2} \dots i_{q_r}},
$$
which, under a possible mode transpose of $\Z$, is
$$
\sum_{\{i_1,i_2,\dots,i_{s}\}\setminus\{i_{q_1},i_{q_2},\dots,i_{q_r}\}}  x^{3}_{i_{j^3_1} i_{j^3_2}\dots i_{j^3_{a_3}}} \dots x^{d}_{i_{j^d_1} i_{j^d_2} \dots i_{j^d_{a_d}}} z_{i_{j^1_1}, i_{j^1_2}\dots,i_{j^1_{b_1}},i_{j^2_1},i_{j^2_2},\dots, i_{j^2_{b_2}}}.
$$
This is the same type of problem for $d-1$ tensors and so the above is no more than $\|\Z\|\prod_{k=3}^d\|\X^k\|$ by induction assumption. Therefore,~\eqref{eq:product} is proved by combining the fact that $\|\Z\|\le\|\X^1\|\cdot\|\X^2\|$ shown earlier.
\end{proof}

We provide some insights of the above result. If $A\in\BF^{n_1\times n_2}$, $B\in\BF^{n_2\times n_3}$ and $C\in\BF^{n_3\times n_1}$ are three matrices, then Lemma~\ref{thm:product2} means that
$$
\tr(ABC) = \sum_{i=1}^{n_1}\sum_{j=1}^{n_2}\sum_{k=1}^{n_3} a_{ij}b_{jk}c_{ki} \le \|A\|\cdot\|B\|\cdot\|\C\|.
$$
As another example, if $\A\in\BF^{n_1\times n_2\times n_3}$, $B\in\BF^{n_1\times n_2}$ and $\bc\in\BF^{n_3}$, then Lemma~\ref{thm:product2} means that
$$
\A\times_{1,2}B\times_3\bc = \sum_{i=1}^{n_1}\sum_{j=1}^{n_2}\sum_{k=1}^{n_3} a_{ijk}b_{ij}c_{k} \le \|\A\|\cdot\|B\|\cdot\|\bc\|.
$$
For both examples, the essential condition is that a same index, such as $i$ and $i$, should be appeared in two different tensors.

Let us return to the spectral norm of a UIT. The condition~\eqref{eq:condition2} in Definition~\ref{def:uit} implies that any $\BI_k$ cannot contain both modes $j$ and $s+j$ of $\II_n$ whose dimensions are both $p_j$.

\begin{proposition}\label{thm:uit}
If $\TT\in\BB^{n_1\times n_2\times \dots \times n_d}$ is a UIT, then $\|\TT\|_\sigma=1$ and $\|\TT\|=\left(\prod_{k=1}^dn_k\right)^{\frac{1}{4}}$. Moreover, $\sqrt{\prod_{k=1}^dn_k}$ is an integer that can be divided by every $n_k$, i.e., the condition~\eqref{eq:condition}.
\end{proposition}
\begin{proof}
Let $\TT=\II_n(\BI_1,\BI_2,\dots,\BI_d)$ where $\{\BI_1,\BI_2,\dots,\BI_d\}$ is a partition of $\{1,2,\dots,2s_n\}$ such that $\left|\BI_k\bigcap\{j,s_n+j\}\right|\le1$ for any $1\le k\le d$ and $1\le j\le s_n$. Recall $\II_n\in\BB^{p_1\times p_2\times \dots \times p_{s_n}\times p_1 \times p_2\times\dots\times p_{s_n}}$ and define $p_{s_n+j}=p_{j}$ for $j=1,2,\dots,s_n$. We have $n_k=\prod_{j\in\BI_k}p_j$ for $k=1,2,\dots,d$. Noticing that the number of entries of $\TT$ is same to that of $\II_n$, one obviously has $\sqrt{\prod_{k=1}^dn_k}=n=\prod_{j=1}^{s_n}p_j$ that must be an integer.

As $\BI_k$ cannot include both $j$ and $s_n+j$, we may define $\BJ_k=\{1\le j\le s_n: j\in\BI_k\mbox{ or } s_n+j\in\BI_k\}\subseteq\{1,2,\dots,s_n\}$ for $k=1,2,\dots,d$. Therefore, $n_k=\prod_{j\in\BI_k}p_j=\prod_{j\in\BJ^k}p_j$ is a factor of $\prod_{j=1}^{s_n} p_j=n$, i.e., $n$ can be divided by $n_k$. For the norms, $\|\TT\|=\|\II_n\|=\sqrt{n}=\left(\prod_{k=1}^dn_k\right)^{\frac{1}{4}}$. It is obvious that $\|\TT\|_\sigma\ge1$ as $\TT$ is a zero-one tensor. It suffices to show that $\|\TT\|_\sigma\le1$.

For every mode $k=1,2,\dots,d$, let $\BI_k=\{j^k_1,j^k_2,\dots,j^k_{a_k}\}$, $\Z^k\in\R^{\times_{j\in\BI_k}p_j}$ be a tensor of order $a_k$, and the vectorization of $\Z^k$ be $\bx^k\in\R^{\prod_{j\in\BI_k}p_j}=\R^{n_k}$. For any $\|\bx^k\|=\|\Z^k\|=1$, it is not hard to see that
\begin{align}
\langle \TT,\bx^1\otimes\bx^2\otimes\dots\otimes \bx^d \rangle
&=\langle \II_n(\BI_1, \BI_2,\dots, \BI_d),\bx^1\otimes\bx^2\otimes\dots\otimes \bx^d \rangle \nonumber \\
&=\langle \II_n,(\Z^1\otimes\Z^2\otimes\dots\otimes \Z^d)^\pi \rangle\nonumber \\
&=\sum_{i_1,i_2,\dots,i_{2s_n}} t_{i_1i_2\dots i_{2s_n}} z^{1}_{i_{j^1_1} i_{j^1_2}\dots i_{j^1_{a_1}}} z^{2}_{i_{j^2_1} i_{j^2_2}\dots i_{j^2_{a_2}}} \dots z^{d}_{i_{j^d_1} i_{j^d_2} \dots i_{j^d_{a_d}}}, \label{eq:summand}
\end{align}
where $\left(\Z^1\otimes\Z^2\otimes\dots\otimes \Z^d\right)^\pi$ denotes a proper mode transpose of $\Z^1\otimes\Z^2\otimes\dots\otimes \Z^d$.

Noticing the relation between $\BJ_k$ and $\BI_k$ and the fact $|\BJ_k|=|\BI_k|$, as $\{\BI_1,\BI_2,\dots,\BI_d\}$ is a partition of $\{1,2,\dots,2s_n\}$, there are exactly two $\BJ_k$'s containing $j$ for every $j=1,2,\dots,s_n$. To avoid new notations, we still denote $\BJ_k=\{j^k_1,j^k_2,\dots,j^k_{a_k}\}$ as that of $\BI_k$ but bare in mind that every element is now the remainder divided by $s_n$. On the other hand, as $\II_n\left(\{1,2,\dots,s_n\},\{s_n+1,s_n+2,\dots,2s_n\}\right)=I_n$, $t_{i_1i_2\dots i_{2s_n}}=1$ if and only if $i_j=i_{s_n+j}$ for all $1\le j\le s_n$, we may remove $i_{s_n+1},i_{s_n+2},\dots,i_{2s_n}$ in the summand of~\eqref{eq:summand} and assign the value of relevant $t_{i_1i_2\dots i_{2s_n}}$ to one, i.e.,
\begin{equation} \label{eq:summand2}
\langle \TT,\bx^1\otimes\bx^2\otimes\dots\otimes \bx^d \rangle
=\sum_{i_1,i_2,\dots,i_{s_n}} z^{1}_{i_{j^1_1} i_{j^1_2}\dots i_{j^1_{a_1}}} z^{2}_{i_{j^2_1} i_{j^2_2}\dots i_{j^2_{a_2}}} \dots z^{d}_{i_{j^d_1} i_{j^d_2} \dots i_{j^d_{a_d}}},
\end{equation}
where $j^k$'s in~\eqref{eq:summand2} denote elements of $\BJ^k$'s while $j^k$'s in~\eqref{eq:summand} denote elements of $\BI_k$'s.

Now, since there are exactly two $\BJ_k$'s containing $j$ for every $j=1,2,\dots,s_n$, by applying Lemma~\ref{thm:product2} to $\Z^k$'s and $\BJ^k$'s, the right hand side of~\eqref{eq:summand2} must be no more than $\prod_{k=1}^d\|\Z^k\|=1$. This shows that $\langle \TT,\bx^1\otimes\bx^2\otimes\dots\otimes \bx^d \rangle\le1$ for any $\|\bx^k\|=1$, i.e., $\|\TT\|_\sigma\le1$.
\end{proof}

From Lemma~\ref{thm:product2}, we do see that any UIT achieves the lower bound $\left(\prod_{k=1}^d n_k\right)^{\frac{1}{4}}$ of~\eqref{eq:main} in Theorem~\ref{thm:main}. This is also true for any slice permutation and multiplication by a positive constant of a UIT, according to Proposition~\ref{thm:permute}. Mode permutation of a UIT is also a case but this is already included in the definition of UIT. Finally, to prove that~\eqref{eq:condition} is sufficient in Theorem~\ref{thm:main}, it suffices to show the existence of a UIT under~\eqref{eq:condition}.

\begin{proposition}
If $d$ integers $n_1,n_2,\dots,n_d\ge2$ such that $\sqrt{\prod_{k=1}^dn_k}$ is an integer that can be divided by $n_k$ for every $k=1,2,\dots,d$, then a UIT exists in $\BB^{n_1\times n_2\times \dots \times n_d}$.
\end{proposition}
\begin{proof}
Let $\sqrt{\prod_{k=1}^dn_k}=n=\prod_{j=1}^{s_n} p_j$ be its prime factorization where $2\le p_1\le p_2\le \dots \le p_{s_n}$. Define $p_{s_n+j}=p_{j}$ for $j=1,2,\dots,s_n$ and $\BJ=\{1,2,\dots,2s_n\}$. One obviously has \begin{equation}
    \prod_{j\in\BJ}p_j=\left(\prod_{j=1}^{s_n} p_j\right)^2=n^2=\prod_{k=1}^dn_k.\label{eq:dimmatch}
\end{equation}

We need to construct a UIT that is unfolded from the $n$th identity tensor $\II_n$. In order to make $\II_n(\BJ_1,\BJ_2,\dots,\BJ_d)\in\BB^{n_1\times n_2\times \dots \times n_d}$ a UIT, it suffices to find a partition $\{\BJ_1,\BJ_2,\dots,\BJ_d\}$ of $\BJ$ with $\left|\BJ_k\bigcap\{j,s_n+j\}\right|\le1$ for all $1\le k\le d$ and $1\le j\le s_n$, such that $\prod_{j\in\BJ_k}p_j=n_k$ for $k=1,2,\dots,d$. In fact, $\BJ_k$ can be defined recursively as
\begin{equation} \label{eq:jk}
\BJ_k = \argmin_{\BI \subseteq \BJ \setminus \bigcup_{i=1}^{k-1}\BJ_i} \left\{ \sum_{j\in\BI} j: \prod_{j\in\BI}p_j=n_k \right\} \mbox{ for } k=1,2,\dots,d.
\end{equation}
Intuitively, we choose indices in $\BJ$ to form $\BJ_1$ via a collection of $p_j$'s whose product is $n_1$, whereas we have multiple choices of $j$ because the $p_j$'s are the same we always choose the smallest available $j$. The elements of $\BJ_1$ are then removed from $\BJ$ and we continue this approach to form $\BJ_2, \BJ_3, \dots, \BJ_d$. Obviously $\{\BJ_1,\BJ_2,\dots,\BJ_d\}$ is a partition of $\BJ$. The feasibility of $\BJ_k$'s in~\eqref{eq:jk} is guaranteed by~\eqref{eq:dimmatch} since
$$
\prod_{k=1}^d \prod_{j\in\BJ_k}p_j = \prod_{k=1}^dn_k = \prod_{j\in\BJ}p_j.
$$
It remains to show that $\left|\BJ_k\bigcap\{j,s_n+j\}\right|\le1$ for all $1\le k\le d$ and $1\le j\le s_n$.

Suppose on the contrary that $\{\ell,s_n+\ell\}\in\BJ_k$ for some $k$ and $1\le \ell\le s_n$.
Let $r$ be the number of primes that are equal to $p_\ell$ among all the prime factors $p_1,p_2,\dots,p_{s_n}$ of $n$. Denote these primes to be $p_i,p_{i+1},\dots,p_{i+r-1}$ where $i\le \ell\le i+r-1$. Thus, $p_i,p_{i+1},\dots,p_{i+r-1},p_{s_n+i},p_{s_n+i+1},\dots,p_{s_n+i+r-1}$ are all the primes that are equal to $p_\ell$ in $\{p_j:j\in\BJ\}$. By the definition of $\BJ_k$, in particular $\sum_{j\in\BJ_k} j$ attaining the minimum, one has $\{\ell,\dots,i+r-1,s_n+i,\dots,s_n+\ell\}\subseteq\BJ_k$ as both $\ell$ and $s_n+\ell$ belong to $\BJ_k$. Thus, $n_k=\prod_{j\in\BJ_k}p_j$ can be divided by $\left(\prod_{j=\ell}^{i+r-1}p_j\right) \left(\prod_{j=s_n+i}^{s_n+\ell}p_j\right)={p_\ell}^{r+1}$. However, $n$ has only $r$ prime factors that are equal to $p_\ell$, contradictory to the fact that $n$ can be divided by $n_k$.
\end{proof}

This concludes the proof of Theorem~\ref{thm:main}.

\subsection{Characterization of extreme tensors}\label{sec:picutre}

The extreme property of unfolded identity tensors for $\R^{n_1\times n_2\times \dots\times n_d}_+$ plays a similar role to that of orthogonal tensors for $\R^{n_1\times n_2\times \dots\times n_d}$ and unitary tensor for $\C^{n_1\times n_2\times \dots\times n_d}$ in~\cite{LNSU18}. Unlike orthogonal and unitary tensors whose existence condition cannot be explicitly characterized, the existence condition for UITs is fully determined by the dimensions, i.e., condition~\eqref{eq:condition}. Under this condition, is any extreme tensor, i.e., a tensor whose ratio between the spectral and Frobenius norms attains the lower bound of~\eqref{eq:main}, a UIT under slice permutation and multiplication of a positive constant?

Unfortunately the answer is no. Consider $\R^{2\times 2\times 2\times 2}_+$ and let $J_4=\left(
  \begin{array}{cccc}
     &  &  & 1 \\
     &  & 1 &  \\
     & 1 &  &  \\
    1 &  &  &  \\
  \end{array}
\right)$ be a permutation matrix, also called a slice permutation of $I_4$. The maximum folding of $J_4$ is $\JJ_4\in\R^{2\times 2\times 2\times 2}_+$, i.e., $\JJ_4(\{1,2\},\{3,4\})=J_4$. Since
$$
1\le \|\JJ_4\|_\sigma\le \|J_4\|_\sigma=1 \mbox{ and } \|\JJ_4\|= \|J_4\|=2,
$$
$\JJ_4$ is an extreme tensor but is neither $\II_4$, nor its slice permutation or mode transpose. There are other examples as well via another slice permutation of $I_4$. The main reason is that the number of slice permutations of $I_4$, $4!$, is more than the number of slice permutations of $\II_4$, which is at most $2^4$. Any slice permutation of a folded tensor can be obtained by a certain slice permutation before the folding, but the reverse is not always possible.

It is straightforward to generalize UITs. Let $I^{(\pi)}_n\in\BB^{n\times n}$ be a permutation matrix where $\pi$ is a permutation of $\{1,2,\dots,n\}$. Let $n=\prod_{k=1}^{s_n} p_k$ be the prime factorization where $2\le p_1\le p_2\le \dots \le p_{s_n}$. Denote $\II^{(\pi)}_n\in\BB^{p_1\times p_2\times \dots \times p_{s_n}\times p_1 \times p_2\times\dots\times p_{s_n}}$ to be the maximum folding of $I^{(\pi)}_n$, i.e.,
$$
\II^{(\pi)}_n\left(\{1,2,\dots,s_n\},\{s_n+1,s_n+2,\dots,2s_n\}\right)=I^{(\pi)}_n.
$$
Here, we use the notation $\II^{(\pi)}_n$ instead of $\II^\pi_n$ as the latter is a mode transpose of $\II_n$ for a permutation $\pi$ of $\{1,2,\dots,2s_n\}$. It is easy to see that $\II^{(\pi)}_n$ is an extreme tensor as $\|\II^{(\pi)}_n\|_\sigma=1$ and $\|\II^{(\pi)}_n\|=\sqrt{n}$.
\begin{definition} \label{def:upt}
Given any permutation $\pi$ of $\{1,2,\dots,n\}$ and any partition $\{\BI_1, \BI_2,\dots, \BI_d\}$ of modes $\{1,2,\dots,2s_n\}$ that satisfies~\eqref{eq:condition2},
$\II^{(\pi)}_n(\BI_1, \BI_2,\dots, \BI_d)$ is called an unfolded permutation tensor (UPT).
\end{definition}

We state the following result that generalizes Proposition~\ref{thm:uit} for UITs. We skip the proof as it involves many index notations resulted by the permutation $\pi$ although the main idea is the same to that of Proposition~\ref{thm:uit}.
\begin{theorem}
If $\TT\in\BB^{n_1\times n_2\times \dots \times n_d}$ is a UPT, then $\|\TT\|_\sigma=1$ and $\|\TT\|=\left(\prod_{k=1}^dn_k\right)^{\frac{1}{4}}$. Moreover, $\sqrt{\prod_{k=1}^dn_k}$ is an integer that can be divided by every $n_k$.
\end{theorem}

Mode transpose of a UPT must be a UPT by Definition~\ref{def:upt} via a permutation of $\{\BI_1,\BI_2,\dots,\BI_d\}$. In fact, a slice permutation of a UPT is also a UPT, originated from another permutation matrix $I^{(\pi')}_n$, i.e., unfolded from another $\II^{(\pi')}_n$. Obviously, multiplication by a positive constant of a UPT must be an extreme tensor as well. We conjecture that these fully characterize the extreme tensors.
\begin{conjecture}\label{thm:con1}
The lower bound~\eqref{eq:main} is achieved by and only by an unfolded permutation tensor up to multiplication by a positive constant.
\end{conjecture}
With Theorem~\ref{thm:main} and an affirmative answer to Conjecture~\ref{thm:con1}, we can conclude a complete story.

Let us work toward the conjecture, albeit we cannot solve it for the time being. Suppose that $\TT\in\R_+^{n_1\times n_2\times \dots\times n_d}$ satisfies $\frac{\|\TT\|_\sigma}{\|\TT\|}=\left(\prod_{k=1}^dn_k\right)^{-\frac{1}{4}}$. Upon multiplying a positive constant, we may assume without loss of generality that $\|\TT\|_\sigma=1$ and $\|\TT\|=\left(\prod_{k=1}^dn_k\right)^{\frac{1}{4}}$. Proposition~\ref{thm:evenly} indicates that $\TT$ must be a zero-one tensor. The maximum folding of $\TT$, keeps the same Frobenius norm and spectral norm of $\TT$ as well since $\TT$ is a zero-one tensor and folding decreases the spectral norm in the weak sense. Therefore, it suffices to consider tensors in the same dimensions to the maximum folding of $\TT$, in other words, every $n_k$ is a prime. To conclude, we need to show that the maximum folding of $\TT$ is $\II^{(\pi)}_n$ or its mode transpose in order to show the correctness of Conjecture~\ref{thm:con1}.
\begin{conjecture}\label{thm:con2}
Let $2\le n=\prod_{k=1}^{s_n} p_k$ be the prime factorization where $2\le p_1\le p_2\le \dots \le p_{s_n}$. If $\TT\in\BB^{p_1\times p_2\times \dots \times p_{s_n}\times p_1 \times p_2\times\dots\times p_{s_n}}$ satisfies $\|\TT\|_\sigma=1$ and $\|\TT\|=\sqrt{n}$, then $\TT$ can be unfolded to a permutation matrix in $\BB^{n\times n}$.
\end{conjecture}

Conjecture~\ref{thm:con2} is true when $n$ is a prime because of Proposition~\ref{thm:evenly}. We also verified this for $n=4,6,9$ with the help of a computer.

\section{Related problems}\label{sec:other}

With the help of the story in Section~\ref{sec:main}, in particular Theorem~\ref{thm:main}, we shall develop various results on the extreme ratios in several contexts.

\subsection{General nonnegative tensors}

The condition~\eqref{eq:condition} in Theorem~\ref{thm:main} immediately implies that any $n_j$ is no more than $\prod_{1\le k\le d,\,k\ne j}n_k$. This is because $\sqrt{\prod_{k=1}^dn_k}$ can be divided by any $n_j$. On the other hand, for a tall tensor where one dimension is very large, i.e., $n_j\ge\prod_{1\le k\le d,\,k\ne j}n_k$ for some $j$, the extreme ratio between the spectral and Frobenius norms can be easily obtained; cf.~\cite[Proposition 2.3]{LNSU18}.
\begin{proposition}\label{thm:tall}
If $d$ positive integers $n_1,n_2,\dots,n_d\ge2$ and $n_j\ge\prod_{1\le k\le d,\,k\ne j}n_k$ for some $1\le j\le d$, then
\begin{equation}\label{eq:tall}
    \phi(\R_+^{n_1\times n_2\times\dots\times n_d})=\left(\prod_{1\le k\le d,\,k\ne j}n_k\right)^{-\frac{1}{2}},
\end{equation}
obtained by and only by a tensor whose mode-$j$ matricization is a submatrix of $I^{(\pi)}_{n_j}$ (permutation matrix), up to multiplication by a positive constant.
\end{proposition}

Theorem~\ref{thm:main} and Proposition~\ref{thm:tall} perfectly match in the intersection where $n_j=\prod_{1\le k\le d,\,k\ne j}n_k$. The extreme ratio keeps the same one $\frac{1}{\sqrt{n_j}}$ and is obtained by and only by a tensor whose mode-$j$ matricization is $I^{(\pi)}_{n_j}$ up to multiplication by a positive constant. For a space other than those $n_k$'s required in Theorem~\ref{thm:main} and Proposition~\ref{thm:tall}, the extreme ratio is generally unknown. However, Theorem~\ref{thm:main} is enough to provide a general idea about the extreme ratio since it includes the case of $n\times n\times \dots \times n$ tensors of order $d$ if $d$ is even or $n$ is a complete square.

\begin{corollary}\label{thm:ordernn}
  For $n\times n\times \dots \times n$ tensors of order $d$ and $n\ge2$, if $d$ is even,
  then
  $$\phi(\R_+^{n\times n\times \dots \times n})=n^{-\frac{d}{4}},$$
  and if $d$ is odd, then
  $$ n^{-\frac{d}{4}}\le  \phi(\R_+^{n\times n\times \dots \times n}) \le \min\big{\{}\left(\sqrt{n+1}-1\right)^{-\frac{d}{2}},n^{-\frac{d-1}{4}}\big{\}}.$$
\end{corollary}
\begin{proof}
The case of even $d$ follows immediately from Theorem~\ref{thm:main}, so as the lower bound of odd $d$.

For the first upper bound of odd $d$, let $p^2\le n \le (p+1)^2-1$ where $p\in\N$. We have that $\sqrt{n+1}-1\le p$. By the monotonicity of the extreme ratio (Lemma~\ref{thm:mono}),
$$
\phi(\R_+^{n\times n\times \dots \times n}) \le
\phi(\R_+^{p^2\times p^2\times \dots \times p^2}) = (p^2)^{-\frac{d}{4}} \le
\left(\sqrt{n+1}-1\right)^{-\frac{d}{2}},$$
where the equality follows Theorem~\ref{thm:main}.

For the second upper bound of odd $d$, again by Lemma~\ref{thm:mono}, the extreme ratio for $n\times n\times \dots \times n$ tensors of order $d$ must be no more than that for $n\times n\times \dots \times n$ tensors of order $d-1$, which is $n^{-\frac{d-1}{4}}$ since $d-1$ is even.
\end{proof}

The first upper bound for odd $d$ nails down the asymptotic order of magnitude for $\phi(\R_+^{n\times n\times \dots \times n})$, which is $O(n^{-\frac{d}{4}})$ no matter $d$ is even or odd. However, this upper bound for odd $d$ can be quite loose, especially for small $n$, under which case the second upper bound is able to compensate.

We remark that the order of magnitude for $\phi(\R_+^{n\times n\times \dots \times n})$ can also be obtained via an example of Theorem~\ref{thm:main} using the norm compression inequality of tensors~\cite[Theorem 5.1]{LZ20}.
\begin{theorem} \label{thm:comp}
  If a nonnegative tensor $\TT\in\R_+^{n_1\times n_2\times \dots \times n_d}$ satisfies $\frac{\|\TT\|_\sigma}{\|\TT\|} = \alpha$, then there exists a nonnegative tensor $\TT_m\in\R_+^{{n_1}^m\times {n_2}^m\times \dots \times {n_d}^m}$ satisfying $\frac{\|\TT_m\|_\sigma}{\|\TT_m\|} = \alpha^m$ for any positive integer $m$. If $\TT$ is further symmetric, then $\TT_m$ is also symmetric.
\end{theorem}
For illustration, there is a nonnegative tensor $\TT\in\R_+^{4\times 4\times \dots \times 4}$ of order $d$ such that $\frac{\|\TT\|_\sigma}{\|\TT\|}=4^{-\frac{d}{4}}$ by Theorem~\ref{thm:main}. Then by Theorem~\ref{thm:comp} there exists a nonnegative tensor $\TT_m\in\R_+^{4^m\times 4^m\times \dots \times 4^m}$ of order $d$ such that $\frac{\|\TT_m\|_\sigma}{\|\TT_m\|}=4^{-\frac{dm}{4}}$ for any positive integer $m$. This provides a general upper bound $O(n^{-\frac{d}{4}})$ for $\phi(\R_+^{n\times n\times \dots \times n})$ if we set $n=4^m$, while the lower bound is obtained by Theorem~\ref{thm:main}. In any case, this confirmed order of magnitude trivially beats the best known upper bound for $\phi(\R_+^{n\times n\times n})$, $O(n^{-0.584})$ in~\cite[Theorem 5.3]{LZ20} whereas ours is $O(n^{-0.75})$ for $d=3$.

In fact, it is not difficult to see that the order of magnitude for general $\phi(\R_+^{n_1\times n_2\times\dots\times n_d})$ is $O\left(\left(\prod_{k=1}^dn_k\right)^{-\frac{1}{4}}\right)$ controlled by Theorem~\ref{thm:comp} with an appropriate example in Theorem~\ref{thm:main}, as long as they are not tall tensors, whose ratio is provided by~\eqref{eq:tall} in Proposition~\ref{thm:tall}. In order to get an explicit upper bound instead of an order of magnitude, we now apply Theorem~\ref{thm:main} again to estimate $\phi(\R_+^{n_1\times n_2\times\dots\times n_d})$ using the power of two.
\begin{theorem}\label{thm:upper}
  If $d$ positive integers $n_1,n_2,\dots,n_d\ge2$ and
   $\max_{1\le k\le d}n_k\le \sqrt{\prod_{k=1}^dn_k}$,
  then
  \begin{equation}\label{eq:upper2}
    \phi(\R_+^{n_1\times n_2\times\dots\times n_d})
    \le 2^{\frac{d+1}{4}}\left(\prod_{k=1}^d n_k\right)^{-\frac{1}{4}}.
  \end{equation}
\end{theorem}
\begin{proof}
Let $2^{a_k}\le n_k<2^{a_k+1}$ where $a_k\in\N$ for $k=1,2,\dots,d$. We estimate the upper bound in three cases below.

If $\R_+^{2^{a_1}\times 2^{a_2}\times \dots \times 2^{a_d}}$ is a space of tall tensors, we may without loss of generality let $\prod_{k=1}^{d-1}2^{a_k}\le 2^{a_d}$. By Lemma~\ref{thm:fold2}, $\phi(\R_+^{n_1\times n_2\times \dots \times n_k})$ is upper bounded by the ratio of its mode-$d$ matricization, $\phi(\R_+^{n_d\times\prod_{k=1}^{d-1} n_k})$, which is equal to $\frac{1}{\sqrt{n_d}}$ since $n_d\le \prod_{k=1}^{d-1}n_k$. Therefore,
$$
\phi(\R_+^{n_1\times n_2\times \dots \times n_k}) \le
\frac{1}{\sqrt{n_d}} \le \left(n_d2^{a_d}\right)^{-\frac{1}{4}} \le
\left(n_d \prod_{k=1}^{d-1}2^{a_k} \right)^{-\frac{1}{4}}
\le \left(n_d \prod_{k=1}^{d-1}\frac{n_k}{2} \right)^{-\frac{1}{4}}
= 2^{\frac{d-1}{4}}\left(\prod_{k=1}^{d}n_k \right)^{-\frac{1}{4}}.
$$

If $\R_+^{2^{a_1}\times 2^{a_2}\times \dots \times 2^{a_d}}$ is not tall and further $\sum_{k=1}a_k$ is even, then by Lemma~\ref{thm:mono} and Theorem~\ref{thm:main}, we have
$$
\phi(\R_+^{n_1\times n_2\times \dots \times n_k}) \le
\phi(\R_+^{2^{a_1}\times 2^{a_2}\times \dots \times 2^{a_d}}) = \left(\prod_{k=1}^d 2^{a_k}\right)^{-\frac{1}{4}} \le
\left(\prod_{k=1}^d \frac{n_k}{2}\right)^{-\frac{1}{4}} = 2^{\frac{d}{4}}\left(\prod_{k=1}^d n_k\right)^{-\frac{1}{4}}.
$$

Finally, if $\R_+^{2^{a_1}\times 2^{a_2}\times \dots \times 2^{a_d}}$ is not tall and $\sum_{k=1}^da_k$ is odd, we need to truncate the largest $2^{a_k}$, say $2^{a_d}$ without loss of generality, by half in the above estimate. This is to keep $\R_+^{2^{a_1}\times 2^{a_2}\times \dots \times 2^{a_d-1}}$ not tall while making $\sum_{k=1}^{d-1} a_k + (a_d-1)$ even. We then have
$$
\phi(\R_+^{n_1\times n_2\times \dots \times n_k}) \le
\phi(\R_+^{2^{a_1}\times 2^{a_2}\times \dots \times 2^{a_d-1}}) = \left(\frac{1}{2}\prod_{k=1}^d 2^{a_k}\right)^{-\frac{1}{4}} \le
\left(\frac{1}{2}\prod_{k=1}^d \frac{n_k}{2}\right)^{-\frac{1}{4}} = 2^{\frac{d+1}{4}}\left(\prod_{k=1}^d n_k\right)^{-\frac{1}{4}}.
$$
The desired upper bound~\eqref{eq:upper2} is proved by combining the three cases.
\end{proof}

As an example for nonnegative tensors of order 3, one has
$$
(n_1n_2n_3)^{-\frac{1}{4}} \le  \phi(\R_+^{n_1\times n_2\times n_3}) \le 2 (n_1n_2n_3)^{-\frac{1}{4}}
$$
if no $n_k$ exceeds the product of the other two.

We conclude this subsection by combining Theorem~\ref{thm:main}, Proposition~\ref{thm:tall} and Theorem~\ref{thm:upper}.
\begin{corollary}\label{thm:upper2}
    If $n_d$ is the largest among $d$ positive integers $n_1,n_2,\dots,n_d\ge2$, then
  $$
    \left(\min\left\{n_d,\prod_{k=1}^{d-1} n_k\right\}\prod_{k=1}^{d-1} n_k\right)^{-\frac{1}{4}}
    \le \phi(\R_+^{n_1\times n_2\times\dots\times n_d})
    \le 2^{\frac{d+1}{4}}\left(\min\left\{n_d,\prod_{k=1}^{d-1} n_k\right\}\prod_{k=1}^{d-1} n_k\right)^{-\frac{1}{4}}.
  $$
\end{corollary}

\subsection{Symmetric tensors}\label{sec:symmetric}

We now study the extreme ratio between the spectral and Frobenius norms in the space of symmetric tensors. By applying homogeneous polynomial mapping discussed in Section~\ref{sec:sym}, we can show that for any $\BF$, $\phi(\BF^{n^d}_\sym)$ is no more than $\phi(\BF^{n\times n\times \dots\times n})$ multiplied by a constant depending only on $d$.

\begin{theorem} \label{thm:sym}
  For any $\BF$ and positive integers $d$ and $n$,
  $$
  \phi(\BF^{dn\times dn\times \dots \times dn}) \le \phi(\BF^{(dn)^d}_\sym) \le  \sqrt{d!d^{-d}}\phi(\BF^{n\times n\times \dots \times n}) \le \sqrt{d!}\phi(\BF^{dn\times dn\times \dots \times dn}).
  $$
\end{theorem}
\begin{proof}
  The lower bound is trivial since $\BF^{n^d}_\sym$ is a subset of $\BF^{n\times n\times \dots \times n}$ for any $n\in\N$. To show the upper bound, let $\TT\in \BF^{n\times n\times \dots \times n}$ such that $$
  \phi(\BF^{n\times n\times \dots \times n}) = \frac{\|\TT\|_\sigma}{\|\TT\|}.$$

  Consider the multilinear form $\langle\TT, \bx^1\otimes \bx^2 \otimes \dots \otimes \bx^d\rangle$ where each $\bx^k$ is a variable vector of dimension $n$. According to Section~\ref{sec:sym}, there is a unique symmetric tensor $\Z\in\BF^{(dn)^d}_\sym$ such that
  $$
  \left\langle\Z, \bx\otimes \bx \otimes \dots \otimes \bx\right\rangle = \langle\TT, \bx^1\otimes \bx^2 \otimes \dots \otimes \bx^d\rangle,
  $$
  where $\bx=\left((\bx^1)^{\T},(\bx^2)^{\T},\dots,(\bx^d)^{\T}\right)^{\T}$ is a variable vector of dimension $dn$. $\Z$ can be partitioned to $d^d$ block tensors in $\BF^{n\times n\times \dots \times n}$ and there are exactly $d!$ nonzero blocks, each of which is equal to $\frac{\TT}{d!}$ or its mode transpose. We thus have
  $$
  \|\Z\|^2 = d!\cdot \frac{\|\TT\|^2}{(d!)^2} = \frac{\|\TT\|^2}{d!}.
  $$
  On the other hand, since $\Z$ is symmetric, it follows by Banach's classical result (Theorem~\ref{thm:banach}) that
  \begin{align*}
  \|\Z\|_\sigma &=\max_{\|\bx\|^2=1} |\langle\Z, \bx\otimes \bx \otimes \dots \otimes \bx\rangle| \\
  & =\max_{\sum_{k=1}^d\|\bx^k\|^2=1} |\langle\TT, \bx^1\otimes \bx^2 \otimes \dots \otimes \bx^d\rangle| \\
  & =\max_{\|\bx^k\|=\frac{1}{\sqrt{d}},\,k=1,2,\dots,d} |\langle\TT, \bx^1\otimes \bx^2 \otimes \dots \otimes \bx^d\rangle| \\
  & =d^{-\frac{d}{2}}\max_{\|\bx^k\|=1,\,k=1,2,\dots,d} |\langle\TT, \bx^1\otimes \bx^2 \otimes \dots \otimes \bx^d\rangle| \\
  & = d^{-\frac{d}{2}}\|\TT\|_\sigma,
\end{align*}
where the third equality is due to $$\left(\prod_{k=1}^d\|\bx^k\|\right)^{\frac{1}{d}}\le\left(\frac{1}{d}\sum_{k=1}^d\|\bx^k\|^2\right)^{\frac{1}{2}}=\frac{1}{\sqrt{d}}$$
and the upper bound is attained only when all $\|\bx^k\|$'s are the same.

Therefore, we obtain
\begin{equation}\label{eq:symupper}
   \phi(\BF^{(dn)^d}_\sym) \le \frac{\|\Z\|_\sigma}{\|\Z\|} = \frac{d^{-\frac{d}{2}}\|\TT\|_\sigma}{(d!)^{-\frac{1}{2}}\|\TT\|} =  \sqrt{d!d^{-d}}\phi(\BF^{n\times n\times \dots \times n}),
\end{equation}
that can generate an upper bound if an asymptotic upper bound of $\phi(\BF^{n\times n\times \dots \times n})$ is available. Even without this information, we can still obtain
$$
   \phi(\BF^{(dn)^d}_\sym) \le \sqrt{d!d^{-d}}\phi(\BF^{n\times n\times \dots \times n}) \le \sqrt{d!d^{-d}}\cdot \sqrt{d}^{\,d} \phi(\BF^{dn\times dn \times \dots \times dn}) = \sqrt{d!} \phi(\BF^{dn\times dn \times \dots \times dn}),
$$
where the last inequality is obtained by applying Lemma~\ref{thm:mono} repeatedly for $d$ times.
\end{proof}

Theorem~\ref{thm:sym} states that the asymptotic order of magnitude for $\phi(\BF^{n^d}_\sym)$ is the same to that for $\phi(\BF^{n\times n\times \dots\times n})$ for any $\BF$. For instance, it was pointed out in~\cite{CKP99} that the order of magnitude for $\phi(\R^{n\times n\times\dots \times n})$ is $O(n^{-\frac{d-1}{2}})$ and so is for $\phi(\R^{n^d}_\sym)$. While Kozhasov and Tonelli-Cueto~\cite{KT22} recently obtained asymptotic upper bounds for both $\phi(\R^{n^d}_\sym)$ and $\phi(\C^{n^d}_\sym)$ by using sophisticated probabilistic analysis, our approach is very simple. In fact, Theorem~\ref{thm:sym} can be used to improve the constant of their estimation. Specifically,~\cite[Theorem 1.1]{KT22} indicates that $\phi(\C^{n\times n\times \dots\times n})\le32\sqrt{d\ln d}\,n^{-\frac{d-1}{2}}$. Applying Theorem~\ref{thm:sym}, we have
\begin{equation}\label{eq:improve}
 \phi(\C^{(dn)^d}_\sym) \le \sqrt{d!d^{-d}}\phi(\C^{n\times n\times \dots \times n}) \le 32\sqrt{d\ln d}\,n^{-\frac{d-1}{2}}\cdot \sqrt{d!d^{-d}} = 32\sqrt{d!\ln d}\,(dn)^{-\frac{d-1}{2}},
\end{equation}
a better estimate than $\phi(\C^{n^d}_\sym)\le 36\sqrt{d!\ln d}\,n^{-\frac{d-1}{2}}$ stated in~\cite[Theorem 1.2]{KT22}, at least when $n$ is a multiple of $d$ or tends to infinity. In any case, the asymptotic order of magnitude for both $\phi(\C^{n^d}_\sym)$ and $\phi(\C^{n^d}_\sym)$ is $O(n^{-\frac{d-1}{2}})$ for fixed $d$.

Let us turn to study $\phi(\R^{n^d}_{+\sym})$. When $d$ is even, we know from Section~\ref{sec:main} that there is a zero-one tensor $\TT$ whose standard matricization is an identity matrix and this $\TT$ is indeed an extreme tensor. However, $\TT$ itself may not be symmetric unless $d=2$. We now provide another construction that only applies to nonnegative tensors.

\begin{theorem} \label{thm:sym2}
If a nonzero tensor $\TT\in\R^{n\times n\times \dots \times n}_+$ and $\sum_\pi\TT^\pi\in\R^{n^d}_{+\sym}$ where the summand is taken over all permutations of $\{1,2,\dots,d\}$, then
$$
\frac{\|\sum_\pi\TT^\pi\|_\sigma}{\|\sum_\pi\TT^\pi\|} \le \sqrt{d!}\frac{\|\TT\|_\sigma}{\|\TT\|}.
$$
As a consequence, one has
\begin{equation}\label{eq:plusbound}
    \phi(\R^{n\times n\times \dots \times n}_+)\le \phi(\R^{n^d}_{+\sym})\le \sqrt{d!}\phi(\R^{n\times n\times \dots \times n}_+).
\end{equation}
\end{theorem}
\begin{proof}
  The number of different permutations of $\{1,2,\dots,d\}$ is $d!$. Any entry of $\sum_\pi\TT^\pi$ is the sum of $d!$ entries of $\TT$. Its square must be larger than or equal to the sum of squares for these $d!$ entries because the square of sum is larger than or equal to the sum of squares for nonnegative numbers. Each entry of $\TT$ appears exactly $d!$ times in $\sum_\pi\TT^\pi$. Therefore, by summing over all the squares for the entries of $\sum_\pi\TT^\pi$, it is easy to see that $\|\sum_\pi\TT^\pi\|^2\ge d!\|\TT\|^2$.

  On the other hand, the triangle inequality implies that $\|\sum_\pi\TT^\pi\|_\sigma\le \sum_\pi \|\TT^\pi\|_\sigma=d!\|\TT\|_\sigma$ by Proposition~\ref{thm:permute}. Combining the two inequalities, we have
  $$
  \frac{\|\sum_\pi\TT^\pi\|_\sigma}{\|\sum_\pi\TT^\pi\|} \le\frac{d!\|\TT\|_\sigma}{\sqrt{d!}\|\TT\|} = \sqrt{d!}\frac{\|\TT\|_\sigma}{\|\TT\|}.
  $$

  It is easy to see that $\sum_\pi\TT^\pi$ represents the generality of tensors in $\R^{n^d}_{+\sym}$. Taking the minimum over all $\TT\in\R^{n\times n\times \dots \times n}_+\setminus\{\OO\}$ leads to the upper bound of~\eqref{eq:plusbound} while its lower bound is trivial.
\end{proof}

Let us now apply Theorem~\ref{thm:sym} and Theorem~\ref{thm:sym2} to get exact estimates of $\phi(\R^{n^d}_{+\sym})$.
\begin{corollary}\label{thm:sym3}
If $d$ is even, then
  $$
  n^{-\frac{d}{4}}\le \phi(\R^{n^d}_{+\sym}) \le \left\{
  \begin{array}{ll}
  d!^{\frac{1}{2}}d^{-\frac{d}{4}}n^{-\frac{d}{4}} & \frac{n}{d}\in\N \\
  d!^{\frac{1}{2}}d^{-\frac{d}{4}}(n+1-d)^{-\frac{d}{4}} & n\ge d \\
  d!^{\frac{1}{2}}n^{-\frac{d}{4}} & n\ge2,
  \end{array}
  \right.
  $$
and if $d$ is odd, then
  $$
  n^{-\frac{d}{4}}\le \phi(\R^{n^d}_{+\sym}) \le \left\{
  \begin{array}{ll}
  d!^{\frac{1}{2}}d^{-\frac{d}{4}} \big{(}\sqrt{n+d}-\sqrt{d}\big{)}^{-\frac{d}{2}} & \frac{n}{d}\in\N \\
  d!^{\frac{1}{2}}d^{-\frac{d}{4}} \big{(}\sqrt{n+1}-\sqrt{d}\big{)}^{-\frac{d}{2}} & n\ge d \\
  d!^{\frac{1}{2}}\min\big{\{}\left(\sqrt{n+1}-1\right)^{-\frac{d}{2}},n^{-\frac{d-1}{4}}\big{\}} & n\ge2.
  \end{array}
  \right.
  $$
\end{corollary}
\begin{proof}
Both lower bounds are obvious by Theorem~\ref{thm:main} and $\phi(\R^{n\times n\times \dots \times n}_+) \le \phi(\R^{n^d}_{+\sym})$. We now focus on the upper bounds.

If $d$ is even, then by Theorem~\ref{thm:sym} and Corollary~\ref{thm:ordernn},
  $$
  \phi(\R^{(dn)^d}_{+\sym}) \le \sqrt{d!d^{-d}}\phi(\R^{n\times n\times \dots \times n}_+) = \sqrt{d!d^{-d}}n^{-\frac{d}{4}} = d!^{\frac{1}{2}}d^{-\frac{d}{4}}(dn)^{-\frac{d}{4}}.
  $$
  To obtain a uniform upper bound for any $m\ge d$, we let $dn\le m\le d(n+1)-1$, implying that $dn\ge m+1-d$. By the monotonicity,
  $$
  \phi(\R^{m^d}_{+\sym})\le \phi(\R^{(dn)^d}_{+\sym})\le d!^{\frac{1}{2}}d^{-\frac{d}{4}}(dn)^{-\frac{d}{4}} \le d!^{\frac{1}{2}}d^{-\frac{d}{4}}(m+1-d)^{-\frac{d}{4}}.
  $$
The last upper bound for even $d$ is immediate from Theorem~\ref{thm:sym2} and Corollary~\ref{thm:ordernn}.

If $d$ is odd, by applying the upper bound $\phi(\R_+^{n\times n\times \dots \times n}) \le \left(\sqrt{n+1}-1\right)^{-\frac{d}{2}}$ in Corollary~\ref{thm:ordernn},
  $$
  \phi(\R^{(dn)^d}_{+\sym}) \le \sqrt{d!d^{-d}}\phi(\R^{n\times n\times \dots \times n}_+) \le \sqrt{d!d^{-d}} \left(\sqrt{n+1}-1\right)^{-\frac{d}{2}} = d!^{\frac{1}{2}}d^{-\frac{d}{4}} \big{(}\sqrt{dn+d}-\sqrt{d}\big{)}^{-\frac{d}{2}}.
  $$
For any $m\ge d$, by letting $dn\le m\le d(n+1)-1$, one also has
  $$
  \phi(\R^{m^d}_{+\sym}) \le \phi(\R^{(dn)^d}_{+\sym}) \le d!^{\frac{1}{2}}d^{-\frac{d}{4}} \big{(}\sqrt{dn+d}-\sqrt{d}\big{)}^{-\frac{d}{2}} \le d!^{\frac{1}{2}}d^{-\frac{d}{4}} \big{(}\sqrt{m+1}-\sqrt{d}\big{)}^{-\frac{d}{2}}.
  $$
  Finally, the last upper bound for odd $d$ is immediate from Theorem~\ref{thm:sym2} and Corollary~\ref{thm:ordernn}.
\end{proof}
The bound obtained by Theorem~\ref{thm:sym2} that works only for $\R_+$ is neat and uniform for all $n$ compared to the bound obtained by Theorem~\ref{thm:sym} that works for any $\BF$, however, with a price of $d^{\frac{d}{4}}$ when $n$ tends to infinity.

\subsection{Extreme ratio between Frobenius and nuclear norms}

We now study the extreme ratio between the Frobenius norm and the nuclear norm. Derksen et al.~\cite{DFLW17} showed that
\begin{equation}\label{eq:nuclear}
\psi(\BF^{n_1\times n_2\times \dots \times n_d})=\phi(\BF^{n_1\times n_2\times \dots \times n_d}) \mbox{ and } \psi(\BF^{n^d}_\sym)= \phi(\BF^{n^d}_\sym) \mbox{ if }\BF=\C,\R,
\end{equation}
and the two extreme ratios can be obtained by the same tensor. With this fact and Theorem~\ref{thm:sym} for symmetric tensors, applying the best estimate of $\phi(\BF^{n\times n\times \dots\times n})$ for $\BF=\C,\R$~\cite[Theorem 1.1]{KT22}, we obtain the following estimates. The proof is similar to the discussion of~\eqref{eq:improve}.
\begin{corollary} \label{thm:nuclearsym}
If $\BF=\C,\R$, then
$$
 n^{-\frac{d-1}{2}}\le \phi(\BF^{n^d}_\sym) = \psi(\BF^{n^d}_\sym) \le\left\{
 \begin{array}{ll}
   32\sqrt{d!\ln d}\,n^{-\frac{d-1}{2}}   & \frac{n}{d}\in\N \\
   36\sqrt{d!\ln d}\,n^{-\frac{d-1}{2}}   & n\ge2.
 \end{array}\right.
 $$
\end{corollary}

However,~\eqref{eq:nuclear} did not close the topic for nonnegative tensors. To our surprise, $\psi(\R_+^{n_1\times n_2\times \dots \times n_d})$ and $\phi(\R_+^{n_1\times n_2\times \dots \times n_d})$ are in general different.
\begin{theorem}\label{thm:nuclear}
If $d$ positive integers $n_1,n_2,\dots,n_d\ge2$, then 
$$
\psi(\R^{n_1\times n_2\times \dots\times n_d})\le \psi(\R^{n_1\times n_2\times \dots\times n_d}_+) \le \sqrt{2}\,\psi(\R^{n_1\times n_2\times \dots\times n_d}),
$$
and if $n\ge2$, then
$$
\psi(\R^{n^d}_\sym)\le \psi(\R^{n^d}_{+\sym}) \le \sqrt{2}\,\psi(\R^{n^d}_\sym).
$$
\end{theorem}
\begin{proof}
  The lower bound is obvious since $\R_+^{n_1\times n_2\times \dots\times n_d}$ is a subset of $\R^{n_1\times n_2\times \dots\times n_d}$. For the upper bound, let $\TT\in\R^{n_1\times n_2\times \dots\times n_d}$ be an extreme tensor for the ratio $\phi(\R^{n_1\times n_2\times \dots\times n_d})$ where $\|\TT\|_\sigma=1$ and $\|\TT\|=\phi(\R^{n_1\times n_2\times \dots\times n_d})^{-1}$.

  Decompose $\TT=\TT_+-\TT_-$ where $\TT_+$ keeps positive entries of $\TT$ and makes other entries zero while $-\TT_-$ keeping negative entries of $\TT$ and makes other entries zero. Obviously, both $\TT_+$ and $\TT_-$ are nonnegative tensors. Since $\|\TT\|^2=\|\TT_+\|^2+\|\TT_-\|^2$, we may assume without loss of generality that $\|\TT_+\|^2\ge\frac{1}{2}\|\TT\|^2$.

  As $\|\TT\|_\sigma=1$, by the dual norm property (Lemma~\ref{thm:dual}), one has $\|\TT_+\|_*\ge \left\langle \TT_+,\TT\right\rangle =\|\TT_+\|^2$. This implies that
  $$
  \psi(\R_+^{n_1\times n_2\times \dots\times n_d}) \le \frac{\|\TT_+\|}{\|\TT_+\|_*}\le \frac{\|\TT_+\|}{\|\TT_+\|^2} =\frac{1}{\|\TT_+\|} \le  \frac{\sqrt{2}}{\|\TT\|} = \sqrt{2}\:\!\phi(\R^{n_1\times n_2\times \dots\times n_d}) = \sqrt{2}\:\!\psi(\R^{n_1\times n_2\times \dots\times n_d}),
  $$
  where the last equality is due to~\eqref{eq:nuclear}.

  The bounds for $\psi(\R^{n^d}_{+\sym})$ can be shown in a similar way by noticing that both $\TT_+$ and $\TT_-$ are symmetric as long as $\TT$ is symmetric.
\end{proof}

Applying the estimates in the literature~\cite{LNSU18,KT22} as well as Theorem~\ref{thm:sym} for symmetric tensors, we are able to nail down the asymptotic order of magnitude for the extreme ratios. The following uniform bounds are obtained using the bounds in~\cite{KT22}, although the constant of the upper bound for $\psi(\R^{n^d}_{+\sym})$ can be slightly improved using Theorem~\ref{thm:sym}, such as that in Corollary~\ref{thm:nuclearsym}.
\begin{corollary}\label{thm:nuclear+}
If $d$ positive integers $n_1,n_2,\dots,n_d\ge2$, then
$$
\frac{1}{\sqrt{\min_{1\le j\le d}\prod_{1\le k\le d,\,k\neq j}n_k}} \le \psi(\R^{n_1\times n_2\times \dots \times n_d}_+)
\le \frac{32\sqrt{2d\ln d}}{\sqrt{\min_{1\le j\le d}\prod_{1\le k\le d,\,k\neq j}n_k}}
$$
and if $n\ge2$, then
$$
n^{-\frac{d-1}{2}} \le \psi(\R^{n^d}_{+\sym}) \le 24\sqrt{2d!\ln d}\,n^{-\frac{d-1}{2}}.
$$
\end{corollary}

Compared with Corollary~\ref{thm:upper2} for $\phi(\R^{n_1\times n_2\times \dots \times n_d}_+)$, the two ratios have different asymptotic order of magnitudes for $d\ge3$, except that for tall tensors where
$\max_{1\le j\le d}n_j \ge \min_{1\le j\le d} \prod_{1\le k\le d,\,k\neq j}n_k$ (this does include the matrix case and vector case), we have
$$
\phi(\BF^{n_1\times n_2\times \dots \times n_d})=\psi(\BF^{n_1\times n_2\times \dots \times n_d})=\frac{1}{\sqrt{\min_{1\le j\le d}\prod_{1\le k\le d,\,k\neq j}n_k}} \mbox{ for any } \BF\supseteq\BB.
$$
For symmetric tensors, $\phi(\R^{n^d}_{+\sym})$ and $\psi(\R^{n^d}_{+\sym})$ are also in different order of magnitudes for $d\ge3$ compared with Corollary~\ref{thm:sym3} while they do be the same in the matrix case and vector case.

\subsection{Low dimensions}

While the extreme ratio for nonnegative tensors is generally understood, it is always a temptation to look into some low dimension cases. In this part we examine $\phi(\R^{n_1\times n_2\times n_3}_+)$ for $2\le n_1,n_2,n_3\le 4$ and $\phi(\R^{n^3}_{+\sym})$ for $2\le n\le 4$.

For $\phi(\R^{n_1\times n_2\times n_3}_+)$, it suffices to check $2\le n_1\le n_2\le n_3\le 4$ because of Lemma~\ref{thm:permute2}. The cases for $\phi(\R^{2\times 2\times 4}_+)=\frac{1}{2}$ and $\phi(\R^{4\times 4\times 4}_+)=\frac{1}{\sqrt{8}}$ are already included in Theorem~\ref{thm:main}, i.e., they satisfy~\eqref{eq:condition}. To obtain $\phi(\R^{2\times 2\times 2}_+)$, we need to use  $\phi(\C^{2\times 2\times 2})=\frac{2}{3}$~\cite{CKP00} as well as the following observation.
\begin{proposition} \label{thm:c+}
$\phi(\C^{n_1\times n_2\times \dots \times n_d})\le \phi(\R^{n_1\times n_2\times \dots \times n_d}_+)$.
\end{proposition}
The main reason is that the definition of the spectral norm for nonnegative tensors remains unchanged by extending to the complex field, i.e., if $\TT\in\R^{n_1\times n_2\times \dots \times n_d}_+$, then
$$
\|\TT\|_\sigma =\max_{\|\bx^k\|=1,\,\bx^k\in\R^{n_k}_+} |\langle \TT, \bx^1\otimes\bx^2\otimes \dots\otimes \bx^d \rangle|  =\max_{\|\bx^k\|=1,\,\bx^k\in\C^{n_k}}|\langle \TT, \bx^1\otimes\bx^2\otimes \dots\otimes \bx^d\rangle|.
$$
With this equivalence, $\R^{n_1\times n_2\times \dots \times n_d}_+$ can be taken as a subset of $\C^{n_1\times n_2\times \dots \times n_d}$ for the optimization problem~\eqref{eq:ratio}, leading to Proposition~\ref{thm:c+}. Therefore, we have $\phi(\R^{2\times 2\times 2}_+)\ge\frac{2}{3}$.

Proposition~\ref{thm:c+} also implies that if a nonnegative tensor achieves $\phi(\C^{n_1\times n_2\times \dots \times n_d})$, then it also achieves $\R^{n_1\times n_2\times \dots \times n_d}_+$. This is true for the space of symmetric tensors as well. In fact, there is a tensor~\cite[Example 5.2]{LZ20} $\TT\in\BB^{2\times 2\times 2}$ whose nonzero entries are $t_{112}$, $t_{121}$ and $t_{211}$ such that $\frac{\|\TT\|_\sigma}{\|\TT\|}=\frac{2}{3}$. Thus, $\phi(\R^{2\times 2\times 2}_+)=\frac{2}{3}$. Since this $\TT$ is symmetric, we also have $\phi(\R^{2^3}_{+\sym})=\frac{2}{3}$.

Currently we are unable to nail down the exact values of $\phi(\R^{n_1\times n_2\times n_3}_+)$ for other small $n_k$'s. We do, however, perform some extensive search over zero-one tensors and obtain the exact values of $\phi(\BB^{n_1\times n_2\times n_3})$ for $2\le n_1\le n_2\le n_3\le 4$. They provide currently the best known upper bounds for $\phi(\R^{n_1\times n_2\times n_3}_+)$, which are believed to be tight. In fact, we would like to make a bold conjecture.
\begin{conjecture}
$\phi(\R^{n_1\times n_2\times \dots \times n_d}_+)= \phi(\BB^{n_1\times n_2\times \dots \times n_d})$.
\end{conjecture}

We summarize exact values or bounds of $\phi(\R^{n_1\times n_2\times n_3}_+)$ for $2\le n_1\le n_2\le n_3\le 4$ in Table~\ref{table:234}. Except for $\phi(\R^{2\times 2\times 2}_+)$, the lower bound is $(n_1n_2n_3)^{-\frac{1}{4}}$ and must not be tight unless~\eqref{eq:condition} is satisfied by Theorem~\ref{thm:main}. The upper bound is exactly $\phi(\BB^{n_1\times n_2\times n_3})$ whose achieved example is also provided.
\begin{table}[ht]
    \centering
    \begin{tabular}{|l|l|l|l|l|}
\hline
$n_1,n_2,n_3$ & Lower bound & $\phi(\BB^{n_1\times n_2\times n_3})$ & Gap & $\TT\in\BB^{n_1\times n_2\times n_3}$ that achieves $\phi(\BB^{n_1\times n_2\times n_3})$ \\
\hline
$2, 2, 2$ & $0.667=2/3$ & $0.667=2/3$ &  & $t_{112}, t_{121}, t_{211}=1$ \\
$2, 2, 3$ & $0.537$ & $0.577=1/\sqrt{3}$ & 0.040 & $t_{111}, t_{212}, t_{223}=1$ \\
$2, 2, 4$ & $0.500$ & $0.500$ &  & $t_{111}, t_{123}, t_{212}, t_{224}=1$\\
$2, 3, 3$ & $0.485$ & $0.500$ & 0.015 & $t_{123}, t_{132}, t_{213}, t_{231}=1$\\
$2, 3, 4$ & $0.452$ & $0.500$ & 0.048 & $t_{114}, t_{132}, t_{213}, t_{222}=1$\\
$2, 4, 4$ & $0.420$ & $0.447=1/\sqrt{5}$ & 0.027 & $t_{113}, t_{121}, t_{142}, t_{214}, t_{231}=1$ \\
$3, 3, 3$ & $0.439$ & 0.469  & 0.030 & $t_{113}, t_{121}, t_{222}, t_{312}, t_{331}=1$ \\
$3, 3, 4$ & $0.408$ & 0.436  & 0.028 & $t_{122}, t_{131}, t_{211}, t_{224}, t_{312},  t_{333}=1$  \\
$3, 4, 4$ & $0.380$ & $0.408=1/\sqrt{6}$  & 0.028 & $
t_{113}, t_{124}, t_{212}, t_{241}, t_{322}, t_{331} =1$ \\
$4, 4, 4$ & $0.354=1/\sqrt{8}$ & $0.354=1/\sqrt{8}$ & &  $t_{111} , t_{123} , t_{231} , t_{243} , t_{312} , t_{324} , t_{432} , t_{444}=1$  \\
\hline
\end{tabular}
\caption{Lower and upper bounds of $\phi(\R^{n_1\times n_2\times n_3}_+)$ for $2\le n_1\le n_2\le n_3\le 4$.} \label{table:234}
\end{table}

For symmetric nonnegative tensors, we summarize similar bounds of $\phi(\R^{n^3}_{+\sym})$ for $2\le n\le 4$ in Table~\ref{table:sym234}. Except for $n=2$ where an exact value is known as mentioned earlier, the lower bounds are the same to that of $\phi(\R^{n\times n\times n}_+)$ in Table~\ref{table:234} and must not be tight. All the upper bounds are from $\phi(\BB^{n^3}_\sym)$.
\begin{table}[ht]
    \centering
    \begin{tabular}{|l|l|l|l|l|}
\hline
$n$ & Lower bound & $\phi(\BB^{n^3}_\sym)$ & Gap & $\TT\in\BB^{n^3}_\sym$ that achieves $\phi(\BB^{n^3}_\sym)$ \\
\hline
$2$ & $0.667=2/3$ & $0.667=2/3$ &  & $t_{112}, t_{121}, t_{211}=1$ \\
$3$ & $0.439$ & 0.471  & 0.032 & $t_{123},t_{132},t_{213},t_{231},t_{312},t_{321}=1$ \\
$4$ & $0.354=1/\sqrt{8}$ & 0.385 & 0.031 &  $t_{123},t_{132},t_{213},t_{231},t_{312},t_{321},t_{344},t_{434},t_{443}=1$  \\
\hline
\end{tabular}
\caption{Lower and upper bounds of $\phi(\R^{n^3}_{+\sym})$ for $2\le n \le 4$.} \label{table:sym234}
\end{table}

\end{document}